%% file: main.tex
\documentclass[onefignum,onetabnum]{siamonline220329}

\Crefname{ALC@unique}{Line}{Lines} %

\input{main_includes}
\newcommand{\norm}[1]{\left\lVert#1\right\rVert}

\ifpdf
\hypersetup{
  pdftitle={Adaptive Uncertainty Quantification for Stochastic Hyperbolic Conservation Laws},
  pdfauthor={J. J. Harmon, S. Tokareva, A. Zlotnik, P. J. Swart}
}
\fi
    
\begin{document}

\maketitle

\begin{abstract}
We propose a predictor-corrector adaptive method for the study of hyperbolic partial differential equations (PDEs) under uncertainty. Constructed around the framework of stochastic finite volume (SFV) methods, our approach circumvents sampling schemes or simulation ensembles while also preserving fundamental properties, in particular hyperbolicity of the resulting systems and conservation of the discrete solutions. Furthermore, we augment the existing SFV theory with \textit{a priori} convergence results for statistical quantities, in particular push-forward densities, which we demonstrate through numerical experiments. By linking refinement indicators to regions of the physical and stochastic spaces, we drive anisotropic refinements of the discretizations, introducing new degrees of freedom (DoFs) where deemed profitable. To illustrate our proposed method, we consider a series of numerical examples for non-linear hyperbolic PDEs based on Burgers' and Euler's equations.
\end{abstract}



\section{Introduction}
Many physical processes that span a variety of settings from water flows in channels to astrophysics can be modeled by hyperbolic partial differential equations (PDEs). Such PDEs often represent systems of balance or conservation laws, where relevant quantities such as mass and momentum are physically conserved. The solution of initial boundary value problems (IBVPs) defined for PDEs of hyperbolic character pose significant computational and theoretical challenges due to discretization requirements and the dynamics of the underlying solution.  The formation of shocks in finite time, which may arise even from globally smooth initial conditions, is particularly problematic for numerical methods.

Given the practical significance of such problems and the inherent inexactness of modeling endeavours, quantifying uncertainty, e.g., due to initial or boundary conditions, driving forces, fluid properties, and other sources remains a vital component of simulation-based design and analysis. Accounting for this uncertainty, however, introduces significant further computational burden. For example, conventional sampling-based approaches address the stochasticity by employing lengthy campaigns of deterministic simulations and Monte Carlo (MC) integration. Such methods offer robustness to the dimension of the stochastic spaces and the stochastic regularity, but lack desirable convergence characteristics as the time-complexity of MC methods typically scales with the square root of the number of samples, and each sample may incur significant computational expense. The importance of this topic and its complexities has compelled significant research investment in the development of multi-level MC (MLMC) methods that aim to significantly improve computational performance \cite{giles2008, hoel2012, giles2015, elfverson2016, eigel2016, krumscheid2018, vanbarel2019}. Major advantages of MLMC methods include their non-intrusiveness, and as a consequence such methods enjoy dominance across many application domains.

Alternatives to MC exchange some scalability in the stochastic dimension for enhanced convergence properties. For example, stochastic Galerkin methods, which, under certain conditions such as stochastic regularity, may achieve exponential rates of convergence \cite{xiu2003}. Situations in which conditions on stochastic regularity are not satisfied may lead to poor or no convergence, depending on the severity of the Gibb's induced oscillations. Moreover, stochastic Galerkin formulations may potentially transform the original stochastic PDE to a set of coupled equations with modified properties. In particular, one risks a potential loss of hyperbolicity that leads to unphysical solutions and barriers to effective numerical implementations. From a software implementation perspective,  and in contrast to MC methods, stochastic Galerkin approaches are highly \textit{intrusive}, in that they require substantial departures from a deterministic solver in addition to more complicated numerical analysis to prove numerical stability and other properties.

Analogous collocation-based methods, which, like stochastic Galerkin approaches, rely on generalized polynomial chaos (gPCE) expansions construct an interpolatory representation of the output uncertainty based on deterministic samples. Sparse grids (e.g., Smolyak) help alleviate the computational burden with respect to increasing stochastic dimensionality \cite{nobile2008, babuska2010, narayan2014, teckentrup2015, barajas2016, jahnke2022}; however, strong sensitivities to the smoothness of the map from the stochastic space to goal functionals or output quantities remain. In the presence of non-smoothness in the stochastic space, global oscillations plague output quantities \cite{barth2013}, inhibiting the practical value for typical stochastic hyperbolic PDEs whether computing low-order statistical moments or push-forward probability densities.

Setting aside for the moment the dichotomy of intrusive or non-intrusive formulations, gPCE-based schemes fundamentally consider a parametric formulation of the uncertainty problem where for proximate realizations in the corresponding stochastic variables, the relationship between the corresponding output quantities is assumed to be smooth. Hyperbolic PDE systems exhibit discontinuous solutions (shocks) that may propagate in both the physical and stochastic spaces \cite{tokareva2014, mishra2016}, which nullifies this assumption. As in the deterministic case, discontinuities in the physical and stochastic spaces may emerge in finite time from smooth initial conditions. Even with a conventional multi-element gPC ansatz, spurious oscillations in the response surface inhibit informative uncertainty quantification \cite{schlachter2020}, though when augmented with appropriate hyperbolicity and stochastic slope limiters, the resulting systems may maintain hyperbolicity and subdue Gibb's oscillations \cite{kusch2022}. While amenable to tuning of the multi-element gPC expansions, significant challenges remain for accurate and efficient capture of solutions to hyperbolic PDEs under uncertainty with these methods.

Even in deterministic settings, \textit{a priori} or manual mesh configurations may neglect vital solution characteristics, lead to poor convergence rates, or exceed computational resource constraints. Addressing these limitations in the deterministic setting using adaptive methods---in which computational resources are allocated by the simulator as the PDE evolves---can result in exponential convergence for certain classes of problems even under conditions of poor solution regularity \cite{babuska1981, gui1986I, gui1986II, gui1986III}. In particular, adaptive mesh refinement can substantially improve computational efficiency in numerous applications from fluid dynamics, even in the presence of shocked flows \cite{berger1984, berger1984b, davis2016, deiterding2016, berger2021}, to fracture propagation in solid media \cite{artina2015, heister2015, lee2016}. Adaptive methods, and the underlying error indication procedures, see significant and continuing research investment. In particular,  \textit{a posteriori} error estimation and adaptive control approaches yield effective simulation tools for reliable uncertainty quantification, whereas otherwise significant and unpredictable bias may result \cite{estep2009, butler2014}. Even for MC-type ensembles, complicated dependencies between random parameters and output quantities require intricate adaptive procedures to drive discretization errors below acceptable levels so as to permit high-confidence design and analysis.

Adaptive numerical methods for stochastic PDEs, on the other hand, have seen substantially less research than adaptive methods in the deterministic setting, in part due to the complexity of obtaining estimates or indications of sensitivity or error contributions and translating them to effective computational resource allocations. Notably, \cite{herty2023} proposes a multiresolution discontinuous Galerkin scheme for solving parameterized stochastic hyperbolic conservation laws based on the multiresolution analysis for deterministic problems \cite{gerhard2016}, with limitations to isotropic refinements.

The development of methods that adaptively resolve uncertainty propagation in higher dimensions in both the physical \textit{and} stochastic spaces through dimension reduction or condensation relies on precise characterization of deterministic and statistical error throughout the discretized domains.
Assuming equal dependencies among the constitutive directions will drive impractically large insertions of degrees of freedom (DoFs), especially for high-dimensional stochastic spaces. While isotropic refinements---which
neglect the directionality of approximation error---simplify implementation, the capabilities unlocked by directional, anisotropic refinement afford a built-in mechanism for reducing
such approximation error precisely. Tailoring resolution in this manner based on the local interactions between the physical and stochastic coordinates as the solution evolves in time promises significant enhancement in the simulation of stochastic PDE systems.

A recent method for non-adaptive stochastic finite volumes (SFV)-based discretization demonstrated significant potential for tractable and scalable simulation of uncertainty propagation in hyperbolic PDE systems on graphs \cite{tokareva2022}. In this study, we expand significantly on SFV methods by  incorporating techniques for automatic distribution of computational resources by adding local degrees of freedom (DoFs) throughout the physical and stochastic spaces. By adapting the discretization via a predictor-corrector system as the PDE evolves, we facilitate an equilibrated, accurate capture of behavior in the physical and stochastic spaces, including in the presence of discontinuous or shocked flows. In contrast to previous studies, we also consider the formal assembly of adaptive discretizations with anisotropic capability that maintain conservation and enhance convergence rates for approximating statistical moments and push-forward densities.

\subsection{Contributions}
\begin{enumerate}[label=(\roman*)]
    \item We propose a novel adaptivity framework underpinned by a flux {1-irregularity} property in Section \ref{sec:refinement_framework}, which in contrast to existing works, permits fully anisotropic refinements while maintaining conservation.
    \item We prove new \textit{a priori} bounds on SFV-based approximations of push-forward probability density and cumulative distribution functions; see Theorems \ref{thm:pdf_error_convergence} and \ref{thm:CDF_convergence_SFV}. We confirm these results through a numerical study of the transport equation with uncertain initial conditions.
    \item From Theorem \ref{thm:CDF_convergence_SFV}, we derive an SFV error indicator, offering a new approach to controlling approximation error in the physical and stochastic spaces under a semi-embedded structure. See Corollary \ref{cor:error_indicator}. We accompany this new error indicator with auxiliary bounds in Proposition \ref{prop:error_bound}
    \item In Algorithm \ref{alg:adaptivity}, we establish a novel predictor-corrector for identifying and controlling approximation error in SFV-type discretizations. The method permits anisotropic or directional refinements throughout the combined physical and stochastic computational domain.
    \item To quantify the performance of our method, we consider challenging benchmark problems based on the stochastic Burgers' and stochastic Euler's equations.
\end{enumerate}

\subsection{Organization}
The remainder of the paper is organized as follows. In Section \ref{sec:prelim}, we outline the stochastic PDEs under consideration and the theoretical underpinnings of our proposed approach. In Section \ref{sec:refinement_framework}, we consider the mesh framework that supports precise adaptivity in the physical and stochastic spaces. Specifically, we construct a conservative discretization framework with improved convergence and computation of approximation solutions and functionals, e.g., statistical moments, with respect to uniform refinements. In Section \ref{sec:error_analysis}, we augment the existing SFV theory with additional \textit{a priori} convergence results for probability density and cumulative distribution function approximations. As a consequence of these results, we propose an error indicator guided by an enriched-reduced reconstruction pair that, when deployed alongside a smoothness estimator, permits robust insertion of new DoFs. Finally, in Section \ref{sec:numerical_results}, we consider the stochastic Burgers' equation and the stochastic Euler equations, demonstrating that our approach delivers significant reductions in the number of DoFs required to attain the same accuracy as high-cost, non-adaptive simulations.

\section{Preliminaries}
\label{sec:prelim}
We begin with the problem of finding a (weak) solution to a hyperbolic system of conservation laws of the form
\begin{equation}\label{eq:deterministic_cons_law}
\mathbf{u}_t + \nabla\cdot \textbf{F}(\mathbf{u}) = \mathbf{0}\; \mathrm{in}\; \Omega_{\mathrm{phys}}\times\mathbb{I},   
\end{equation}
where $\Omega_{\mathrm{phys}}\subset \mathbb{R}^d$  and $\mathbb{I}$ is a finite subset of $\mathbb{R}_{>0}$ that represents the temporal region of interest for conserved variables $\mathbf{u}\in \mathbb{R}^p$ that are subject to hyperbolic fluxes $\mathbf{F}\in\mathbb{R}^{p\times d}$ given appropriate initial and boundary conditions. Subscripts $(\cdot)_t$ indicate a partial derivative with respect to time $t$, and $\nabla$ acts on the spatial coordinates alone.

Many PDEs of interest, under specification of $\mathbf{F}$, exhibit the structure of \eqref{eq:deterministic_cons_law}. Relevant examples to our study include the scalar \textit{transport} or \textit{advection} equation, with 
\begin{equation}\label{eq:advection} F_{\mathrm{adv}} = a u,\quad a\in\mathbb{R}, 
\end{equation} 
for a wave speed $a$; \textit{Burgers'} equation with 
\begin{equation}\label{eq:burgers_flux} F_{\mathrm{Burgers}} = \frac12 u^2, 
\end{equation} 
and \textit{Euler's} equation, with 
\begin{align}\label{eq:euler_flux}
    \mathbf{F}_{\mathrm{Euler}} &= \begin{bmatrix}
                                        \rho v \\
                                        \rho v^2 + p \\
                                        v(\rho E + p)
                                        \end{bmatrix},
\end{align}
where $\rho$, $v$, $p$, and $E$ represent the density, velocity, pressure, and total energy of the flow. In the case of Euler's equation, the vector of conserved variables is 
\begin{align}
    \mathbf{u} &= \begin{bmatrix}
                                        \rho \\
                                        \rho v \\
                                       \rho E
                                        \end{bmatrix},
\end{align}
where $\rho E$ depends on an auxiliary equation of state (EOS), such as the ideal gas EOS or tabulated varieties.

The fundamental challenge of seeking solutions to \eqref{eq:deterministic_cons_law} together with associated initial and boundary conditions stems from the general absence of analytical solutions except under specific, highly restrictive conditions. In this study, we thus consider \textit{weak solutions}, which satisify equation \eqref{eq:deterministic_cons_law} in some integral sense. Such weak solutions are \textit{not} unique, however. Selecting in this family of solutions the physically ``correct" or relevant one requires the concept of \textit{entropy} solutions or \textit{viscosity limiting} solutions \cite{toro2009}.

Furthermore, we extend the hyperbolic conservation law \eqref{eq:deterministic_cons_law} together with associated initial and boundary conditions to the stochastic setting.  Consider now a probability space $\left(\Omega_{\mathrm{stoch}}, \, \mathcal{F},\, \mathbb{P} \right)$, composed of a set $\Omega_{\mathrm{stoch}}$, a $\sigma$-algebra $\mathcal{F}$, and a probability measure $\mathbb{P}$ on $\mathcal{F}$.
Let $\mathbf{y} = \mathbf{y}(\omega)\in D_{\mathrm{stoch}}\subset\mathbb{R}^q$ measurable $\mathcal{F}$ denote a random variable on $\left(\Omega_{\mathrm{stoch}}, \, \mathcal{F},\, \mathbb{P} \right)$. Assume there exists for $\mathbf{y}$ a probability density $\mu$ such that $\mathbb{P}(A) = \int_A \mu(\mathbf{y})\, d\mathbf{y}$ for all $A\in \mathcal{F}$. We further assume the expectation $\mathbb{E}$ of $\mathbf{y}$ is finite, that is,
\begin{equation}
    \mathbb{E}[\mathbf{y}] = \int_\Omega \mathbf{y}\mathbb{P}(d\mathbf{y}) = \int_{\mathbb{R}^q} \mathbf{y}\mu(\mathbf{y})\, d\mathbf{y} < \infty,
\end{equation}
and for $u(\cdot;\, \mathbb{Y}(\omega))\in L^1(\Omega_{\mathrm{stoch}}; \, \mathbb{R}^q),$
\begin{equation}
    \mathbb{E}[u] = \int_{\mathbb{R}^q} u\left(\cdot; \, \mathbf{y}\right)\mu\left(\mathbf{y}\right) d\mathbf{y} < \infty,
\end{equation}
where $L^1(\Omega_{\mathrm{stoch}}; \, \mathbb{R}^q)$ denotes the associated Bochner space.

With this setting, let us generalize the deterministic conservation law \eqref{eq:deterministic_cons_law} to one with uncertainty. Consider a hyperbolic system of conservation laws dependant on events $\omega\in\Omega_{\mathrm{stoch}}$ with uncertain flux
\begin{equation}\label{eq:stochastic_cons_law}
    \mathbf{u}_t + \nabla\cdot\mathbf{F}(\mathbf{u};\, \omega) = \mathbf{0} \; \mathrm{in}\; \Omega_{\mathrm{phys}}\times\mathbb{I},\, \omega\in\Omega_{\mathrm{stoch}},
\end{equation}
subject to random initial data 
\begin{equation}\label{eq:stochastic_ini_con}
    \mathbf{u}(\mathbf{x}, 0; \omega) = \mathbf{u}_0(\mathbf{x},\omega),\quad \mathbf{x}\in \Omega_{\mathrm{phys}},\; \omega\in \Omega_{\mathrm{stoch}},
\end{equation}
and random boundary conditions
\begin{equation}\label{eq:stochastic_boundary_con}
    \mathbf{u}(\mathbf{x},t;\omega) = \mathbf{u}_B(t, \omega),\quad \mathbf{x}\in\partial\Omega_{\mathrm{phys}},\; \omega\in\Omega_{\mathrm{stoch}}.
\end{equation}
For scalar conservation laws in multiple dimensions, existence and uniqueness of random entropy solutions to the stochastic IBVP \eqref{eq:stochastic_cons_law}-\eqref{eq:stochastic_boundary_con} was proven in \cite{mishra2012}. In the interest of clarity of our exposition, we consider \eqref{eq:stochastic_cons_law} in the absence of an excitation.  The approach for solving the stochastic IBVP in a more general setting with a non-zero source term, which could be deterministic or stochastic, is largely unchanged.

Under the assumptions above, we parameterize the random inputs in the stochastic IBVP \eqref{eq:stochastic_cons_law}-\eqref{eq:stochastic_boundary_con} via the random variable $\mathbf{y} = \mathbf{y}(\omega)$, leading to a \textit{parametric} system of hyperbolic conservation laws,
\begin{equation}\label{eq:stochastic_cons_law_param}
    \mathbf{u}_t + \nabla\cdot\mathbf{F}(\mathbf{u}, \, \mathbf{y}) = 0, \; \mathbf{x}\in\Omega_{\mathrm{phys}},\, \mathbf{y}\in D_{\mathrm{stoch}},
\end{equation}
subject to
\begin{align}
 \mathbf{u}(\mathbf{x},0,\mathbf{y}) = \mathbf{u}_0(\mathbf{x},\mathbf{y}),&\quad \mathbf{x}\in \Omega_{\mathrm{phys}},\; \mathbf{y}\in D_{\mathrm{stoch}}, \label{eq:stochastic_ini_con_param} \\
    \mathbf{u}(\mathbf{x},t,\mathbf{y}) = \mathbf{u}_B(t,\mathbf{y}),&\quad \mathbf{x}\in\partial\Omega_{\mathrm{phys}},\; \mathbf{y}\in D_{\mathrm{stoch}}. \label{eq:stochastic_boundary_con_param}
\end{align}

For all $t\in \mathbb{I}$, we assume $\mathbf{u}$ is Bochner integrable. Let \begin{equation} 
\mathcal{A}(\mathbf{u}, \mathbf{v};\, \mathbf{y}) = 0, \quad \forall \mathbf{v}\in\mathbb{V}
\end{equation} 
denote the weak form of \eqref{eq:stochastic_cons_law_param}-\eqref{eq:stochastic_boundary_con_param} for a fixed $\mathbf{y}\in D_{\mathrm{stoch}}$ with respect to an appropriate test space $\mathbb{V}$.
Suppose a subdivision or triangulation of the stochastic space $D_{\mathrm{stoch}}$ by cells $T_{y}\subset D_{\mathrm{stoch}}$. To define the SFV scheme, we let 
\begin{equation}
\mathbb{E}[\mathbf{u}\,|\,\mathbf{y}\in T_y] = \frac{1}{P\left(\mathbf{y}\in T_y\right)}\int_{T_y} \mathbf{u}\mu(\mathbf{y})\,d\mathbf{y}, 
\end{equation}
and then consider finding a weak solution $\mathbf{u}$ such that
\begin{equation}\label{eq:general_SFV}
    \mathbb{E}[\mathcal{A}(\mathbf{u}, \mathbf{v};\, \mathbf{y})\, | \, \mathbf{y}\in T_y] = 0, \,\forall \mathbf{v}\in\mathbb{V},
\end{equation}
for each $T_y$.

\begin{remark}\label{remark:unbounded}
If $D_{\mathrm{stoch}}$ is an unbounded subset of $\mathbb{R}^q$, as expected for, e.g., Gaussian random variables, we first partition $D_{\mathrm{stoch}}$ such that $$ D_{\mathrm{stoch}} = D_{\mathrm{stoch}}^{\mathrm{bounded}}\cup D_{\mathrm{stoch}}^{\mathrm{unbounded}},\;D_{\mathrm{stoch}}^{\mathrm{bounded}}\cap D_{\mathrm{stoch}}^{\mathrm{unbounded}}=\emptyset,  $$
and $\mathbb{P}(\mathbf{y}\in D_{\mathrm{stoch}}^{\mathrm{unbounded}}) \le \epsilon.$ The above procedure is then applied to the bounded subset $D_{\mathrm{stoch}}^{\mathrm{bounded}},$ neglecting the error induced by truncating the stochastic space. Terminating the partition at boundaries of the domain with infinite cells of finite measure is also feasible.
\end{remark}

Such an outer scheme is independent of the discretization in the physical space and may augment continuous or discontinuous Galerkin finite element methods, finite volume methods, and so on \cite{tokareva2013}. Specifying, for example, \eqref{eq:general_SFV} with a full finite volume approximation, i.e., in both the physical and stochastic spaces, with the physical space subdivided into cells $T_x$ results in the following scheme:
\begin{equation}\label{eq:sfv_fv_scheme}
    \frac{1}{P\left(\mathbf{y}\in T_y\right)}\left( \int_{T_x} \int_{T_y} \mathbf{u}_t\mu(\mathbf{y})\,d\mathbf{x}d\mathbf{y} + \int_{T_x} \int_{T_y} \nabla\cdot\mathbf{F}(\mathbf{u}, \mathbf{y})\mu(\mathbf{y})\, d\mathbf{x}d\mathbf{y} \right) = 0.
\end{equation}

Let $T \coloneqq T_x\times T_y$ denote the control volume over a subset of $\Omega_{\mathrm{phys}}\times D_{\mathrm{stoch}}$. For each such volume, let $h_T$ denote its measure, i.e.,
        \begin{equation}
            h_T = |T_x||T_y| = \int_{T_x}\int_{T_y} \mu(\mathbf{y})\,d\mathbf{x}d\mathbf{y},
        \end{equation}
        where $|T_x|$ and $|T_y|$ represent the physical and stochastic volumes of cell $T$, respectively.
        
        In the conventional manner, introducing the cell average
        \begin{equation}
        {\mathbf{U}}_T = \frac{1}{h_T}\int_{T_x}\int_{T_y} \mathbf{u}(\mathbf{x},\,t,\,\mathbf{y})\mu(\mathbf{y})\,d\mathbf{x}d\mathbf{y}
        \end{equation}
        in the SFV scheme \eqref{eq:sfv_fv_scheme} leads after partial integration over $T_x$ to the ODE system
        \begin{equation}\label{eq:sfv_ODE}
            \frac{d{\mathbf{U}}_T}{dt} + \frac{1}{h_T}\int_{\partial T_x} \int_{T_y} \left( \mathbf{F}(\mathbf{u}, \mathbf{y})\cdot\hat{\mathbf{n}}\right) \mu(\mathbf{y})\, d\mathbf{x}d\mathbf{y} = 0.
        \end{equation} 
        Here, and elsewhere in the study, variations on $\mathbf{U}$ refer to DoFs, and  particularly cell averages, while $\mathbf{u}$ refers explicitly to point values.
        
        Note that \eqref{eq:sfv_ODE} is still exact, describing instead the evolution of the cell averages. However, in direct symmetry with the conventional deterministic problem, we replace the potentially discontinuous flux through the cell boundary, $\mathbf{F}(\mathbf{u}, \mathbf{y})\cdot\hat{\mathbf{n}}$, by any suitable numerical flux approximation, such as the Lax-Friedrichs or Rusanov flux \cite{toro2009}. With such a substitution, applying a suitable time integration scheme such as a Runge-Kutta method solves the ordinary differential equation system of \eqref{eq:sfv_ODE}. For spatially discretized PDEs of hyperbolic character, typical choices include the strong stability preserving (SSP) Runge-Kutta (RK) methods \cite{gottlieb2001}. In this manuscript, we assume a sufficiently small global time-step such that the error resulting from the temporal discretization may be neglected. Multi-rate or local time stepping schemes, such as in the typical deterministic framework of Berger and Oliger \cite{berger1984}, may afford additional efficiency. Combining adaptivity in the physical and stochastic spaces with temporal adaptivity is a compelling topic for future work.
        
        We now recall several key results.
        
        \begin{theorem}[(7.11), \cite{tokareva2014}]\label{thm:solution_L1}
         Let $u$ denote the exact solution to the stochastic IBVP \eqref{eq:stochastic_cons_law}-\eqref{eq:stochastic_boundary_con}, and let $u^y$ be the numerical solution that is exact in $\mathbf{x}$ and discretized in $\mathbf{y}$ of order $\ell$. Similarly, let $u^{xy}$ denote the numerical solution discretized in both $\mathbf{x}$ and $\mathbf{y}$ of order $r$ and $\ell$, respectively. Assume the following holds:
         \begin{align}
             \norm{u^y - u^{xy}}_{L^1} &\le C_1 |T_x|^{r}, \quad \forall \mathbf{y}\in D_{\mathrm{stoch}},\\
             \norm{u -u^y}_{L^1} &\le C_2|T_y|^{\ell}, \quad \forall \mathbf{x}\in\Omega_{\mathrm{phys}}.
         \end{align}
         Then, we have that
         \begin{equation}
             \norm{u-u^{xy}}_{L^1} \le C_1|T_x|^{r} + C_2|T_y|^{\ell}.
         \end{equation}
        \end{theorem}
        
        \begin{theorem}[(7.17), \cite{tokareva2014}]\label{thm:expectation_L1} Let the assumptions of Theorem \ref{thm:solution_L1} hold. The approximation error of the expected value is bounded such that 
        \begin{equation}
         \norm{\mathbb{E}[u] - \mathbb{E}[u^{xy}]}_{L^1}\le C_1|T_x|^r + C_2|T_y|^{\ell}.
         \end{equation}
        \end{theorem}
        
        For sufficiently smooth solutions, Thm. \ref{thm:solution_L1} and Thm. \ref{thm:expectation_L1} may be stated analogously for the $L^\infty$-norm.

\section{Refinement Framework}
\label{sec:refinement_framework}
We now outline the mathematical framework for describing domain refinement in the adaptive method for approximating solutions to the stochastic IBVP \eqref{eq:stochastic_cons_law}-\eqref{eq:stochastic_boundary_con} via the parameterization \eqref{eq:stochastic_cons_law_param}. We propose a refinement framework that  emphasizes balance laws to ensure in practice that refinements in the physical and stochastic spaces do in fact conserve the theoretically conserved variables throughout the evolution of the PDE.  

Denote by $\mathbb{T}$ an admissible set of discretizations. For each $\mathcal{T}\in\mathbb{T}$ there exists $\left\{ \mathcal{T}_0,\dots,\mathcal{T}_n \right\}$ such that
        \begin{equation}
            \mathcal{T} = \bigcup_{\ell = 0}^{n-1} \mathcal{T}_{\ell},
        \end{equation}
        where each mesh $\mathcal{T}_\ell$ is composed of a non-overlapping, potentially incomplete covering of $\Omega_{\mathrm{phys}}\times D_{\mathrm{stoch}}$ via conformal control volumes $K_{\ell_i}, \, i\in\mathcal{I}_\ell,$ where $\mathcal{I}_\ell$ denotes the index set for level $\ell$.  Keeping in mind that further requirements on each mesh $\mathcal{T}_\ell$ will be detailed below, we note that while $\mathcal{T}_0$ must form a complete covering, $\mathcal{T}_\ell,\,\ell > 0$ \textit{may} form a complete covering. In other words, though $\mathcal{T}_\ell$ is conformal, the union $\mathcal{T}$ is not.

        Considering the separation of the physical and stochastic spaces, note that we have non-zero fluxes only in physical directions.  Letting $\mathcal{F}(T)$ denote the faces of $T$ and letting $\mathcal{F}_{x}(T)$ denote the faces of $T_x$, we write $x\in f_x$ for $x\in\mathbb{R}^d$ along $f_x\in\mathcal{F}_x(T)$.
        
        For each $f_x\in{F}_{x}$, let $f_y(f_x)\in\mathcal{F}(T)$ denote the associated face such that $x\in f_x$ for all $(x,\, y)\in f_y(f_x)$. We denote by $\mathcal{F}_y$ the set of all such faces, i.e.,
        \begin{equation}
        \mathcal{F}_y(T) = \left\{ f_y(f_x) \, | \, f_x\in\mathcal{F}_x(T) \right\}.
        \end{equation}
        Furthermore, let $\mathrm{refine}\left( \mathcal{T} \right)\subset\mathbb{T}$ denote the set of all admissible refinements $\mathcal{T}_*\in\mathbb{T}$ of $\mathcal{T}$.
        
        \begin{remark}
        Let $\mathcal{T}\in\mathbb{T}$ and $\mathcal{T}_*\in\mathrm{refine}(\mathcal{T})$. Then there exists $m>0$ such that $$ \mathcal{T}_* = \mathcal{T} \cup \left( \bigcup_{\ell = n}^{n+m-1} \mathcal{T}_\ell \right). $$
        \end{remark}
        
        \begin{remark}
        Let $\mathcal{T}\in\mathbb{T}$ and take $K_{\ell_i},\, 0<\ell \leq n$. Then there exists $j\in\mathcal{I}_k$, $0\leq k \leq n,\, k \neq \ell$ such that $$\mathrm{interior}\left(K_{\ell_i}\right)\cap\mathrm{interior}(K_{k_j})\neq \emptyset.$$ In other words, any refined cell must be a descendant of a cell on a previous level.
        \end{remark}
        
        Our construction of $\mathbb{T}$ implies a \textit{hierarchical} structure to the generation of new, refined discretizations. As such, we must outline a concept of \textit{activity}. Consider a forest $\mathbb{F}$ of $\mathcal{T}$. Naturally, for each $K\in\mathcal{T}$, we have $K\in\mathbb{F}$. Moreover, $\mathrm{root}(K) = K_{0_i}\in\mathcal{T}_0$ for some $i\in\mathcal{I}_0$.  
        \begin{definition}[Forest]
        A forest $\mathbb{F}$ of $\mathcal{T}$ is a disjoint union of inheritance trees for each origin cell $K_0\in\mathcal{T}_0.$ 
        \end{definition}

        \begin{definition}[Active cell]
         A cell $K_{\ell_j}$ is considered \textit{active} iff $K_{\ell_j}$ is a leaf of $\mathbb{F}$.
        \end{definition}
        
        \begin{remark}
        Suppose that $K_{\ell_i}$ is an active cell.  Then for all $K_{k_j}$, $K_{k_j}\neq K_{\ell_i}$ such that 
        \begin{equation}
            \mathrm{interior}\left(K_{\ell_i}\right)\cap\mathrm{interior}(K_{k_j})\neq \emptyset,
        \end{equation}
        if $K_{k_j}$ is inactive.
        \end{remark}
        We use $\hat{\mathcal{T}}$ to refer to the set of all active cells in $\mathcal{T}$.
        
        Given the structure of the problems described in Section \ref{sec:prelim}, we outline a final set of constraints on the relationship between each $K\in\mathcal{T}$. Note that the augmented problem \eqref{eq:sfv_ODE} features fluxes only in the physical directions, with purely parametric stochastic coordinates that are not inherent to the PDE problem. In this manner, conservation of physical quantities is maintained through the treatment of these fluxes at the interfaces $\mathcal{F}_y(K)$. To help enforce this conservation, we constrain all $\mathcal{T}\in\mathbb{T}$ to \textit{flux 1-irregularity}, which we define below.

        \begin{definition}[Flux 1-Irregularity]
        \label{def:irregularity}
         A mesh $\mathcal{T}$ is considered Flux 1-Irregular if it satisfies for all $T\in\hat{\mathcal{T}}$ and for all $f_y\in \mathcal{F}_y(T)$, that if there exists $T'\in\hat{\mathcal{T}}$ such that $f_y$ and $f'_y\in\mathcal{F}_y(T')$ intersect with non-zero measure, then $T$ and $T'$ share \textit{at least} $q$ vertices, where $q$ is the dimension of the stochastic space.
        \end{definition}
        
        \begin{figure}[]
        \centering
       
       \begin{subfigure}[b]{0.75\textwidth}
        \includegraphics[width=\textwidth]{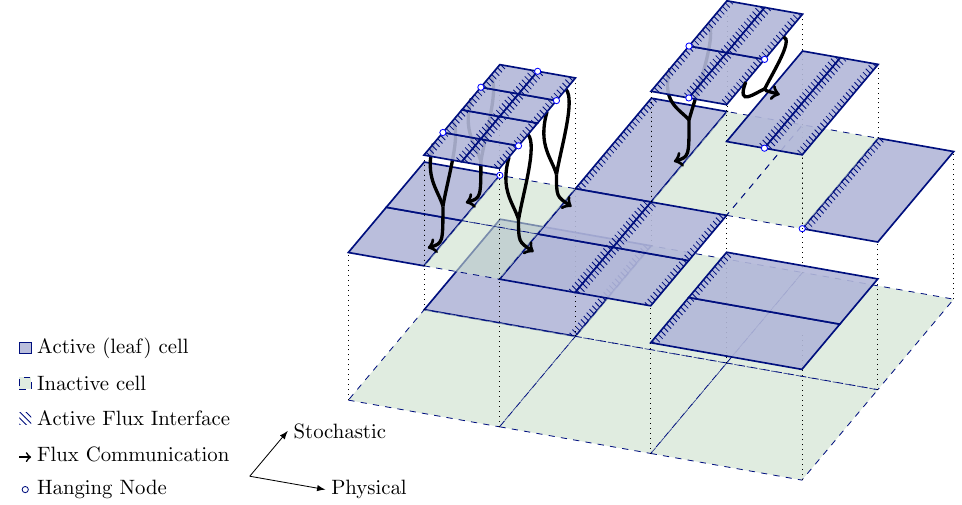}
        \caption{}
    \end{subfigure}
    \begin{subfigure}[b]{0.22\textwidth}
        \includegraphics[width=\textwidth]{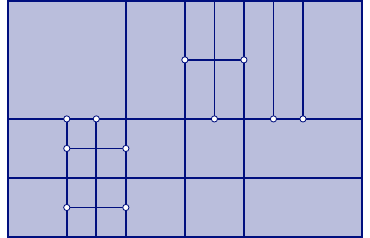}
        \caption{}
    \end{subfigure}
    
        \caption{An admissible discretization that satisfies the flux 1-irregularity and hierarchy constraints. (a) A multi-level perspective that illustrates the hierarchical refinement structure. (b) A 2-D perspective of the same discretization in (a). Notice that in the stochastic space, the proposed approach permits multiple levels of hanging nodes.}
        \label{fig:example_discretizations}
\end{figure}

         An example discretization that belongs to $\mathbb{T}$ is illustrated in Fig. \ref{fig:example_discretizations}. Note that this does \textit{not} restrict refinements at interfaces $f\in\mathcal{F}(T)\setminus \mathcal{F}_y(T)$. \textit{Anisotropic} refinements, i.e., directional refinements, are also permissible \textit{subject to the constraint of flux 1-irregularity}. This imposition of flux 1-irregularity constraints facilitates conservation in isotropic or anisotropic refinements when refinement levels differ between adjacent cells and the flux integrals must be communicated. Fig. \ref{fig:example_discretizations}(a) depicts this general procedure, with the active flux interfaces and necessary flux communications shown. Note that in a case of a coarser cell adjacent to a finer cell, the flux integral is taken as the merged flux integrals of the two finer cells. Flux 1-irregularity, in some sense, is not \textit{necessary}, but affords a simplified, robust ruleset for ensuring conservation that is analogous to similar requirements in the context of $hp$ finite elements \cite{bangerth2011}.
        
We summarize the constraints imposed on $\mathcal{T}_\ell,\, \mathcal{T},$ and $\mathbb{T}$ below:
        \begin{enumerate}[label=(\roman*)]
        \item If $K_{\ell_i},\, K_{\ell_j},\, i\neq j$ are two control volumes on the $\ell$th level, then $K_{\ell_i}\cap\,K_{\ell_j}$ is either empty, a vertex, edge, face, etc.;
        \item If $K_{\ell_i},\, 0<\ell \leq n$, then there exists $j\in\mathcal{I}_k$, $0\leq k \leq n,\, k \neq \ell$ such that $\mathrm{interior}\left(K_{\ell_i}\right)\cap\mathrm{interior}(K_{k_j})\neq \emptyset$;
        \item If $K_{\ell_i}$ is an active cell, then for all $K_{k_j}$, $K_{k_j}\neq K_{\ell_i}$ that $\mathrm{interior}\left(K_{\ell_i}\right)\cap\mathrm{interior}(K_{k_j})$ is not the empty set, $K_{k_j}$ is inactive;
        \item $\mathcal{T}\in\mathbb{T}$ is at most 1-irregular in the sense of Definition \ref{def:irregularity}.
        \end{enumerate}
        
        The first three conditions above can be restated as: (i) hanging nodes are not permitted on the same level; (ii) any refined cell must be a descendant of a cell on a previous level; and (iii) if a descendant cell is active, then its ancestors must not be, and vice versa.

        From this point on, to simplify notation by $\mathcal{T}\in\mathbb{T}$, we refer to the leaves of $\mathbb{F}$ rather than the full hierarchy unless otherwise stated. However, the inheritance structure is relevant in the case of coarsening as is discussed in Section \ref{sec:coarsening}. While the approach outlined is largely agnostic to the choice of discretization, e.g., finite volume, finite element---continuous or discontinuous, we assume a finite volume ansatz in our exposition, i.e., piece-wise constant, unless otherwise specified.
       
        Now, suppose $\mathcal{T}\in\mathbb{T}$ is a refinement of some $\mathcal{T}_0$. Let $\mathbb{V}(\mathcal{T})$ denote a finite dimensional space generated by $\mathcal{T}$. In the context of finite volume methods, $U(\mathcal{T})\in\mathbb{V}(\mathcal{T})$ denotes a piecewise-constant construction on $\mathcal{T}$.

 We define the following reconstruction operator:
\begin{definition}[Reconstruction] A reconstructor $\mathbb{B}_{U(\mathcal{T})}(T; \xi):\mathbb{V}(\mathcal{T})\rightarrow\mathbb{R}^p$ maps a solution $U(\mathcal{T})\in\mathbb{V}(\mathcal{T})$ associated with DoFs on $\mathcal{T}$ to point values $\xi\in T$.
\end{definition}
Even under a piecewise constant ansatz, a reconstruction may result in, e.g., a higher-order, piecewise polynomial response. The \textit{conservative} reconstructions considered in our study include the ubiquitous essentially non-oscillatory (ENO) or weighted essentially non-oscillatory (WENO) reconstructions \cite{shu1998}. Note that while we have fluxes only in the physical directions, reconstruction includes point evaluations in the combined physical and stochastic spaces. Note also that by $\mathbb{B}$ we refer to a reconstruction of generic order, and by $\mathbb{B}^{{l}}$ an ${l}$th order accurate reconstruction in each direction, i.e.,
\begin{equation}\label{eq:reconstruction_order}
    \mathbb{B}^{l}_{U(\mathcal{T})}(T; \boldsymbol{\xi}) = \mathbf{u}(\xi) + O(\Delta\boldsymbol{\xi}^{l + 1}\cdot\mathbf{1}),\quad \boldsymbol{\xi}\in T,
\end{equation}
under conditions of sufficient smoothness, where $\Delta\boldsymbol{\xi}$ represents the vectorial diameter of the cell $T$. Moreover, for two reconstructions of order $\mathrm{H}$ and $\mathrm{L}$, $ \mathrm{L} < \mathrm{H} $, we assume there exists $C>0$ independent of $\Delta \boldsymbol{\xi}$ such that
\begin{equation}\label{eq:reconstruction_diff}
\left\lvert\mathbb{B}^{\mathrm{H}}_{U(\mathcal{T})}(T; \boldsymbol{\xi}) - \mathbb{B}^{\mathrm{L}}_{U(\mathcal{T})}(T; \boldsymbol{\xi}) \right\rvert \le C(\Delta \boldsymbol{\xi}^{L+1}).
\end{equation}

For adaptive schemes that include refinement and coarsening, we require notions that relate numerical solutions on $\mathcal{T}\in\mathbb{T}$ and $\mathcal{T}_*\in\mathbb{T}$:

\begin{definition}[Conservative Projection]\label{def:cons_projection}
A projection $\Pi_{\mathcal{T}} U(\mathcal{T}_*):\mathbb{V}(\mathcal{T}_*)\rightarrow\mathbb{V}(\mathcal{T})$ that maps $U(\mathcal{T}_*)$ to $\mathbb{V}_{\mathcal{T}}$
such that $\Pi_{\mathcal{T}}$ minimizes $\norm{U(\mathcal{T}_*) - \Pi_{\mathcal{T}}U(\mathcal{T}_*)}$ subject to the constraint
\begin{equation} \sum_{T_*\in\mathcal{T}_*} \int_{T_*} \mathbb{B}_{U(\mathcal{T}_*)}(T_*; \boldsymbol{\xi})\mu\left( \mathbf{y}\left(\boldsymbol{\xi}\right) \right)\,d\boldsymbol{\xi} =  \sum_{T\in\mathcal{T}}\int_{T} \mathbb{B}_{\Pi_{\mathcal{T}}U(\mathcal{T}_*)}(T; \boldsymbol{\xi})\mu\left( \mathbf{y}\left(\boldsymbol{\xi}\right) \right)\,d\boldsymbol{\xi}
\end{equation}
is a conservative projection.
\end{definition}

For $T_*\equiv T$, the projection $\Pi_{\mathcal{T}}$ is simply the identity mapping.

 \begin{definition}[Conservative Local Interpolation]\label{def:cons_interp}
 Let $\mathcal{T}\in\mathbb{T}$ and take $U(\mathcal{T})\in\mathbb{V}(\mathcal{T})$. Consider $\mathcal{T}_*\in\mathrm{refine}(\mathcal{T})$ and $\mathcal{M}_T\subset\mathcal{T}_*$ such that for all cells $T_*\in\mathcal{M}_T$, $T\in\mathcal{T}$ is the parent of $T_*$. Then, any $I_T(U(\mathcal{T}))$ that interpolates $U(\mathcal{T})$ using a finite nearest neighbor set of $T$ to $U(\mathcal{M}_T)$ such that
 $$ \int_{T} \mathbb{B}_{U(\mathcal{T})}(T; \xi)\mu\left( \mathbf{y}\left(\boldsymbol{\xi}\right) \right)\,d\boldsymbol{\xi} =  \sum_{T_*\in\mathcal{M}_T}\int_{T_*} \mathbb{B}_{I_T(U(\mathcal{T}))}(T_*; \boldsymbol{\xi})\mu\left( \mathbf{y}\left(\boldsymbol{\xi}\right) \right)\,d\boldsymbol{\xi},$$
 is a conservative local interpolant.
 \end{definition}

\begin{remark}
The constraints on $\Pi_{\mathcal{T}}$ and $I_T$ ensure that the mapping between approximation spaces maintains conservation (in the absence of conservation violating boundary conditions or excitations).
\end{remark}

We consider below a specific example of an interpolant as part of a full finite volume framework.
\begin{proposition}
Let $T\in\mathcal{T}$ and $\mathcal{M}_T\subset\mathcal{T}_*$ be as in Definition \ref{def:cons_interp}. A first-order interpolant
\begin{equation}\label{eq:interpolant} 
I_T(U(\mathcal{T});\, T_*) = U(T) + \tilde{\nabla}^1U(T)\cdot \left( T_*^{\mathrm{c}} - T^{\mathrm{c}} \right), \quad \forall T_*\in\mathcal{M}_T,
\end{equation}
where $\tilde{\nabla}^1(\cdot)$ denotes a limited finite difference (e.g., \texttt{minmod} or \texttt{van Leer}, \cite{vanleer1979}) using the first nearest neighbors of $T$, and $T_*^{{\mathrm{c}}}$ is the probabilistic center of the cell $T_*$, i.e.,
\begin{align}
    T_*^{{\mathrm{c}}} &= \begin{bmatrix}
                                        |{T_*}_x|^{-1}\int_{{T_*}_x} \mathbf{x}\, d\mathbf{x} \\
                                        \mathbb{E}[\mathbf{y}\,|\, \mathbf{y}\in {T_*}_y]
                                        \end{bmatrix},
\end{align}
is a conservative local interpolant in the sense of Definition \ref{def:cons_interp}.
\begin{proof}
Assume a conservative reconstruction $\mathbb{B}$. Let $T^{\mathrm{c}}$ denote the probabilistic center of cell $T$ and let $\tilde{\nabla}^1 U(T)$ represent the \texttt{van Leer} gradient. In the neighborhood of $T^{\mathrm{c}},$ we have the first-order interpolant
\begin{equation}\label{eq:f_interpolant} \mathcal{J}(\boldsymbol{\xi}) = U(T) + \tilde{\nabla}^1 U(T)\cdot (\boldsymbol{\xi} - T^{\mathrm{c}}).\end{equation}

Seeking the interpolated cell averages from \eqref{eq:f_interpolant}, we integrate over $T_*\in\mathcal{M}_T$ with respect to the stochastic measure to obtain
\begin{dmath}
I_T(U(\mathcal{T});\, T_*) = h^{-1}_{T_*}\int_{T_*} \mathcal{J}(\boldsymbol{\xi})\mu\left(\mathbf{y}\left(\boldsymbol{\xi}\right)\right)\, d\boldsymbol{\xi} = U(T) + h^{-1}_{T_*}\int_{T_*} \tilde{\nabla}^1 U(T)\cdot (\boldsymbol{\xi} - T^{\mathrm{c}})\mu\left(\mathbf{y}\left(\boldsymbol{\xi}\right)\right)\, d\boldsymbol{\xi} = U(T) + \tilde{\nabla}^1 U(T)\cdot h^{-1}_{T_*}\int_{T_*} (\boldsymbol{\xi} - T^{\mathrm{c}})\mu\left(\mathbf{y}\left(\boldsymbol{\xi}\right)\right) \, d\boldsymbol{\xi}  = U(T) + \tilde{\nabla}^1U(T)\cdot \left( T_*^{\mathrm{c}} - T^{\mathrm{c}} \right),
\end{dmath}
which yields the interpolant \eqref{eq:interpolant}.
For conservative $\mathbb{B},$
\begin{equation}
h^{-1}_{T_*}\int_{T} \mathbb{B}_{I_T(U(\mathcal{T}))}(T_*; \boldsymbol{\xi})\mu\left(\mathbf{y}\left(\boldsymbol{\xi}\right) \right)\,d\boldsymbol{\xi} = I_T\left(U(\mathcal{T});\, T_*\right),
\end{equation}
i.e., the interpolant preserves averages, and likewise for $U(\mathcal{T})$. 
Hence, we obtain
\begin{equation}\sum_{T_*\in\mathcal{M}_T} h_{T_*} I_T\left(U(\mathcal{T});\, T_*\right) = \int_{T} \mathcal{J}(\boldsymbol{\xi})\mu\left(\mathbf{y}\left(\boldsymbol{\xi}\right)\right)\, d\boldsymbol{\xi} = h_T U(T)
\end{equation}
for the interpolant.
\end{proof}
\end{proposition}

\section{Error Analysis and Adaptivity}\label{sec:error_analysis} With the mathematical foundations for adaptive discretizations established, we investigate the \textit{a priori} convergence behavior of the SFV method in Sections \ref{sec:convergence} and \ref{sec:ex:transport} below. Following those developments, we construct the pillars of our adaptive scheme---including refinement, coarsening, and equilibration---in Sections \ref{sec:refinement} through \ref{sec:equilibration}, on which we sculpt a predictor-corrector adaptive scheme for studying hyperbolic PDEs under uncertainty in Section \ref{sec:predictor_corrector}. 
\subsection{Convergence Results for the SFV Method}\label{sec:convergence}
We begin by considering the push-forward density $f$ of a solution $u$ that evolves as a result of the underlying uncertainty at a space-time coordinate $(\mathbf{x},\, t)$. In the following exposition, we assume that $u$ is scalar; the general, vector-valued case follows from an otherwise identical component-wise treatment.
\begin{theorem}[SFV Method PDF Approximation Error]\label{thm:pdf_error_convergence}
 Suppose $f$ is sufficiently smooth. Let $\hat{f}$ denote the approximation to $f$ via the stochastic finite volume method of order $(r,\,\ell)$ with physical diameter $\Delta_x$ and stochastic diameter $\Delta_y$. Then the approximation error induced by the discretization is bounded such that
 \begin{displaymath}
 \norm{f-\hat{f}} \le M \left(\Delta_x^r + \Delta_y^{\ell}\right) + \mathcal{R}, \quad  \mathcal{R},\,M >0,
 \end{displaymath}
 where $M>0$ depends on $n,\, h,$ but not on $\Delta_x$ or $\Delta_y$, and $\mathcal{R}$ is a remainder with respect to $\Delta_x$, $\Delta_y$, $n$, and $h$.
\end{theorem}
\begin{proof}
  Smoothness of $f$ implies that for all $\varepsilon > 0$ there exists $n>N$ such that
  $$\norm{f - f_K} \le \varepsilon,$$
  where $$f_K(u_\phi) = \frac{1}{nh} \sum_{i=1}^n K\left(\frac{u_\phi - u_i}{h}\right),$$ for a smooth kernel $K$, bandwidth $h$, and exact samples $u_i$ of $u$ at $\left( \mathbf{x}, t\right).$

  Let $\hat{f} = \frac{1}{nh} \sum_{i=1}^n K(\frac{u_\phi - \hat{u}_i}{h})$ denote the associated SFVM-based approximation via $\hat{u}$.
  Then \begin{equation}
      \norm{f-\hat{f}} = \norm{f - f_K + f_K - \hat{f}} \le \norm{f-f_K} + \norm{f_K-\hat{f}}.
  \end{equation}
  By smoothness of $f$, we have that
  \begin{align}\label{eq:sampling_error}
      \norm{f-f_K} &\le \varepsilon,\\
      \implies \norm{f-\hat{f}} &\le \varepsilon + \norm{f_K-\hat{f}}.
  \end{align}
  Expanding $\norm{f_K-\hat{f}}$ results in
  \begin{dmath}
      \norm{f_K-\hat{f}} = \frac{1}{nh} \norm{ \sum_{i = 1}^n K\left(\frac{u_\phi - u_i}{h}\right) - K\left(\frac{u_\phi - \hat{u}_i}{h}\right)}
      \le \frac{1}{nh}\sum_{i=1}^n \norm{K\left(\frac{u_\phi - u_i}{h}\right) - K\left(\frac{u_\phi - \hat{u}_i}{h}\right)}
  \end{dmath}
  
    Smoothness of $K$ implies the existence of an expansion
    \begin{equation}
        K(u) = K(u_0) + (u-u_0)K'(u_0) + O((u-u_0)^2)
    \end{equation}
      \begin{align}
          \implies \norm{K\left(\frac{u_\phi - u_i}{h}\right) - K\left(\frac{u_\phi - \hat{u}_i}{h}\right)} &= \norm{M\left(\frac{\hat{u}_i}{h} - \frac{{u}_i}{h}\right) + O((\hat{u}_i - {u}_i)^2)} \\
          &\le C_i \left(\Delta_x^r + \Delta_y^{\ell}\right) + \mathcal{R},
      \end{align}
      for a remainder $\mathcal{R} > 0$, where we have applied, under the assumption of smoothness, the uniform convergence of $\hat{u}_i$ to $u_i$ implied by the $L^\infty$-norm equivalent of Thm. \ref{thm:solution_L1}.
      
      Hence,
      \begin{dmath}\label{eq:SFVM_error}\norm{f_K-\hat{f}}_\chi \le  \frac{1}{nh} \sum_{i=1}^n C_i \left(\Delta_x^r + \Delta_y^{\ell}\right) + \mathcal{R} \le \frac{1}{nh} \sum_{i=1}^n \sup_i(C_i) \left(\Delta_x^r + \Delta_y^{\ell}\right) + \mathcal{R} \le M \left(\Delta_x^r + \Delta_y^{\ell}\right) + \mathcal{R},\end{dmath}
      where $M>0$ depends on $n$ and $h$ but not on $\Delta_x$ or $\Delta_y$.
      Combining \eqref{eq:sampling_error} and \eqref{eq:SFVM_error} yields
      \begin{equation}
          \norm{f-\hat{f}} \le M \left(\Delta_x^r + \Delta_y^{\ell}\right) + \mathcal{R},
      \end{equation}
      which completes the proof.
\end{proof}
Note that under the conditions of \cref{thm:pdf_error_convergence}, we obtain exponential convergence with respect to SFVM order.

We now state the following theorem for the convergence of a \textit{simplified estimator}, which considers the stochasticity of the averages. As opposed to the point-wise results above, we consider this an \textit{intrinsic} estimator, in the sense that it only employs piecewise constant data. As we will see, this representation reveals an inexpensive refinement indicator.
\begin{theorem}[SFVM CDF Approximation Error]\label{thm:CDF_convergence_SFV}
 Let $F(g;\,x,\,t) = P\left(U(x,\,t) \le g\right)$ denote the \textit{exact} cumulative distribution function (CDF) of the cell averages, and by $\hat{{F}}$ its SFV approximation. Then the approximation error induced by the SFV discretization is bounded such that
 \begin{displaymath}
 \norm{{F}-\hat{{F}}_{\mathrm{SFVM}}}_{L^1} \le C_0\left(\Delta_x^r\right) + C_1\left(\Delta_y + \Delta_y^{\ell}\right), 
 \end{displaymath}
 where $\norm{\cdot}_{L^1}\equiv\norm{\cdot}_{L^1(\Omega_{\mathrm{phys}}\times\mathbb{R})}$.
\end{theorem}
\begin{proof}
  For every $x\in\Omega_{\mathrm{phys}}$ there exists an index set $\mathcal{I}(x)$ associated with the cells $K_i,\, i\in\mathcal{I}(x),$ in the computational domain $\Omega_{\mathrm{phys}}\times D_{\mathrm{stoch}}$ such that
 $$ x\in K_i^{\mathrm{phys}},\quad\sum_{i\in\mathcal{I}(x)} P(y\in K_i^{\mathrm{stoch}}) = 1.$$ We construct the approximation $\hat{{F}}_{\mathrm{SFVM}}$ via
 $$ \hat{{F}}_{\mathrm{SFVM}}(g) = \sum_{i\in\mathcal{I}(x)} P \left(y\in K_i^{\mathrm{stoch}} \right)\mathbbm{1}(\hat{{U}}_i \le g).$$
 
 Let ${F}_{\mathrm{SFVM}}$ denote the equivalent representation but with exact averages. Then the $L^1$ approximation error satisfies
 \begin{dmath}\label{eq:expanded_CDF_error} \norm{{F}-\hat{{F}}_{\mathrm{SFVM}}}_{L^1} = \norm{{F}-{F}_{\mathrm{SFVM}} + {F}_{\mathrm{SFVM}} - \hat{{F}}_{\mathrm{SFVM}}}_{L^1} \le  \norm{{F}-{F}_{\mathrm{SFVM}}}_{L^1} + \norm{{F}_{\mathrm{SFVM}}-\hat{{F}}_{\mathrm{SFVM}}}_{L^1}.\end{dmath}
 The first term in \eqref{eq:expanded_CDF_error} represents the error introduced by treating the variation in the stochastic space according to volumes with exact averages, as opposed to point values, \textit{and treating that data as equivalent to the point valued data}, that is, assuming the solution is constant within a given volume. This is in major contrast to the reconstructed, point-wise data convergence results derived above. The discrepancy of this term is related to the convergence of averages in $D_{\mathrm{stoch}}$ to point values, i.e.,
 \begin{equation}\label{eq:CDF_volume_error} \norm{{F} - {F}_{\mathrm{SFVM}}}_{L^1} \le C \Delta_y.\end{equation}

 Expanding the second term on the right-hand-side of \eqref{eq:expanded_CDF_error} leads to
 \begin{align} \label{eq:before_splitting_error}\norm{{F}_{\mathrm{SFVM}}-\hat{{F}}_{\mathrm{SFVM}}}_{L^1} &= \norm{\sum_{i\in\mathcal{I}(x)} P \left(y\in K_i^{\mathrm{stoch}} \right)\left(\mathbbm{1}({U}_i \le g) - \mathbbm{1}(\hat{{U}}_i \le g)\right)}_{L^1} \\ &= \norm{\sum_{i\in\mathcal{I}(x)} P \left(y\in K_i^{\mathrm{stoch}} \right)\left(\mathbbm{1}(\hat{{U}}_i + \varepsilon_i \le g) - \mathbbm{1}(\hat{{U}}_i \le g)\right)}_{L^1}.\end{align}
Modifying the arguments of \cite[Thm. 4]{chaudhry2021} results in
 
 \begin{align}\norm{{F}_{\mathrm{SFVM}} - \hat{{F}}_{\mathrm{SFVM}}}_{L^1} &\le \norm{\sum_{i\in\mathcal{I}(x)} P \left(y\in K_i^{\mathrm{stoch}} \right)\left(\mathbbm{1}(\hat{{U}}_i - |\varepsilon_i| \le g \le \hat{{U}}_i + |\varepsilon_i|)\right)}_{L^1} 
 \\ &\le \int_{\Omega_{\mathrm{phys}}} \sum_{i\in\mathcal{I}(x)} \int_G P \left(y\in K_i^{\mathrm{stoch}} \right)\left(\mathbbm{1}(\hat{{u}}_i - |\varepsilon_i| \le g \le \hat{{u}}_i + |\varepsilon_i|)\right)\, dx dg
 \\ &= \int_{\Omega_{\mathrm{phys}}} \sum_{i\in\mathcal{I}(x)}P(K_i^{\mathrm{stoch}})2|\varepsilon_i| d\, x \label{eq:to_error_estimator}
 \\ &\le C(\Delta_x^r + \Delta_y^{\ell}).\label{eq:second_term_CDF_avs}
 \end{align}
 Combining \eqref{eq:CDF_volume_error} and \eqref{eq:second_term_CDF_avs}, we observe that
 \begin{equation}
     \norm{{F}-\hat{{F}}_{\mathrm{SFVM}}}_{L^1} \le C_0\left(\Delta_x^r\right) + C_1\left(\Delta_y + \Delta_y^{\ell}\right).
 \end{equation}
\end{proof}
\begin{remark}
The implications of Thm. \ref{thm:CDF_convergence_SFV} are twofold. First, lower orders in the stochastic space may be employed without significant impact on these computed quantities, which confirms previously established experimental results \cite{tokareva2013}. Second, recovering high-order rates of convergence in the stochastic space requires application of the reconstruction procedure outlined above in the stochastic space. 
\end{remark}

\begin{corollary}[A computable refinement indicator]\label{cor:error_indicator}
For the CDF, it follows immediately from \eqref{eq:to_error_estimator} that an error estimate for the cell $T$ can be given by
\begin{equation}\label{eq:computable_refinement_indicator}
    \eta_T = C_{\mathrm{ref}}h_{T}\left\lvert\varepsilon_T\right\rvert,\quad C_{\mathrm{ref}} > 0,
\end{equation}
which, though computed only from the volume averages, approaches zero at least as fast as the full, reconstructed solution values.
\end{corollary}

Observe that a local notion of approximation error 
\begin{equation}
    \left\lvert \varepsilon_T \right\rvert = \left\lvert{U}_T - \hat{U}_T\right\rvert
\end{equation}
is implicit in the computation of \eqref{eq:computable_refinement_indicator}, even though the exact averages ${U}_T$ are of course unavailable. We must therefore substitute approximations to compensate for this general limitation. For unsteady problems, the question is whether an estimate of the approximation error at the current time step or at a future time step offers greater utility. In the former case, we might estimate $\left\lvert\varepsilon_T\right\rvert$ by
\begin{equation} 
\left\lvert\varepsilon_T\right\rvert \approx \int_{T} \left\lvert\mathbb{B}^{\mathrm{H}}_{U(\mathcal{T})}(T; \xi) - \mathbb{B}^{\mathrm{L}}_{U(\mathcal{T})}(T; \xi)\right\rvert \mu\left( \mathbf{y}\left(\boldsymbol{\xi}\right) \right)\,d\boldsymbol{\xi} 
\end{equation}
as a measure of the local resolution via the variation between two reconstructions, e.g., one of high-order $\left( \mathbb{B}^{\mathrm{H}} \right)$ and one of low-order $\left( \mathbb{B}^{\mathrm{L}} \right)$. Depending on the choice of $\mathrm{H}$ and $\mathrm{L}$ relative to the computation and evolution of the cell-averaged system \eqref{eq:sfv_ODE}, the same asymptotic behavior as the error may be lost, potentially resulting in misplaced refinement. In other words, if $\mathbb{B}^{\mathrm{H}}$ and $\mathbb{B}^{\mathrm{L}}$ are too low-order to capture the character of the true approximation error,  the error indicators may lag the reduction in error, driving more refinements than necessary. However, for hyperbolic problems, high rates of convergence may generally be obtained only locally as a result of shock or discontinuity formations, so that coarser estimators may be preferred or at least not precluded.

Note that for conservative reconstructions, taking instead $\left\lVert\varepsilon_T\right\rVert$ as 
\begin{equation}
  \left\lvert \int_{T} \left( \mathbb{B}^{\mathrm{H}}_{U(\mathcal{T})}(T; \xi) - \mathbb{B}^{\mathrm{L}}_{U(\mathcal{T})}(T; \xi)\right ) \mu\left( \mathbf{y}\left(\boldsymbol{\xi}\right) \right)\,d\boldsymbol{\xi}\right\rvert
\end{equation}
is non-informative (zero) regardless of $\mathrm{H}$ and $\mathrm{L}$.

Now, consider the latter approach of error approximation at a future time step. Keeping in mind the method of lines framework of the cell-averaged system \eqref{eq:sfv_ODE}, denote by $\mathbb{K}(\cdot)$ the action of a generic time integrator. To simplify notation, we use $U_{n}$ to denote the solution at time $t_n$. Applying $\mathbb{K}$ yields the approximate solution at $t_{n+1}$, i.e.,
\begin{equation}\label{eq:time_integrator}
    U_{n+1}\leftarrow \mathbb{K}\left(U_n\right),
\end{equation}
explicitly or implicitly. According to this perspective, we may readily choose
\begin{equation}\label{eq:partial_step_error}
    \left\lVert\varepsilon_T\right\rVert \approx \left\lVert U^{\mathrm{H}}_{n+1}(T) - U^{\mathrm{L}}_{n+1}(T) \right\rVert
\end{equation}
for computing the refinement indicator in \eqref{eq:computable_refinement_indicator}. Nothing prohibits choosing $\mathbb{K}$ distinct from the evolution of the cell-averaged ODE system $\eqref{eq:sfv_ODE}$. In fact, for strong stability preserving (SSP) schemes \cite{ketcheson2008}, which deploy a convex combination of the first-order Euler steps, we may estimate the error after the first step in the higher-order scheme. In this sense, we have an embedded estimator.

\begin{proposition}\label{prop:error_bound}
Let $\mathbf{u}$ satisfy the conditions of Thm. \ref{thm:solution_L1} and Thm. \ref{thm:expectation_L1}. Furthermore, let $\mathbf{F}(\mathbf{u};\, \mathbf{y})$ be Lipschitz for all admissible $\mathbf{y}$, i.e., 
\begin{equation}
    \left\lvert \mathbf{F}(\mathbf{u}^L;\, \mathbf{y}) - \mathbf{F}(\mathbf{u}^H;\, \mathbf{y}) \right\rvert \le C_{\mathbf{F}}\left\lvert \mathbf{u}^L - \mathbf{u}^H\right\rvert.
\end{equation}

Suppose $\mathbb{K}$ represents the first-order Euler step
\begin{equation}
    \mathbf{U}_{n+1}(T) \leftarrow \mathbf{U}_n(T) - \frac{\Delta_t}{h_T}\int_{\partial T_x} \int_{T_y} \left( \mathbf{F}(\mathbf{u}, \mathbf{y})\cdot\hat{\mathbf{n}}\right) \mu(\mathbf{y})\, d\mathbf{x}d\mathbf{y},
\end{equation}
and the local approximation error is evaluated by
$\left\lvert\varepsilon_T\right\rvert \approx \left\lvert \mathbb{K}(\mathbf{U}^H_n(T)) - \mathbb{K}(\mathbf{U}^L_n(T)) \right\rvert$.
Then, the refinement indicator \eqref{eq:computable_refinement_indicator} on cell $T$ satisfies
\begin{equation*}
    \eta_T \le C\Delta_t \left\lvert T_y \right\rvert \left\lvert \partial T_x \right\rvert(\Delta \boldsymbol{\xi})^{\mathrm{L}+1},\quad C > 0,
\end{equation*}
where the constant $C$ depends on the Lipschitz condition on $\mathbf{F}$, $C_{\mathrm{ref}}$ in \eqref{eq:computable_refinement_indicator}, and the reconstruction orders \eqref{eq:reconstruction_order}, but not the measure $h_T$.
\begin{proof}
By the definition of $\eta_T$,
\begin{align}
    \eta_T &= C_{ref}\Delta_t \left\lvert \int_{\partial T_x}\int_{T_y} \left[ \mathbf{F}(\mathbf{u}^L;\,\mathbf{y}) - \mathbf{F}(\mathbf{u}^H;\,\mathbf{y}) \right]\cdot\hat{\mathbf{n}}\mu(\mathbf{y})\,d\mathbf{x}d\mathbf{y} \right\rvert \\
    &\le C_{ref}\Delta_t \int_{\partial T_x} \int_{y} \left\lvert \mathbf{F}(\mathbf{u}^L;\,\mathbf{y}) - \mathbf{F}(\mathbf{u}^H;\,\mathbf{y}) \right\rvert \mu(\mathbf{y})\, d\mathbf{x}d\mathbf{y} \\
    &\le C' \Delta_t \int_{\partial T_x} \int_{T_y} \left\lvert \mathbf{u}^L - \mathbf{u}^H \right\rvert \mu(\mathbf{y})\, d\mathbf{x}d\mathbf{y},
\end{align} where we applied the Lipschitz condition on $\mathbf{F},$ merging constants.

Notice the point values $\mathbf{u}^{\mathrm{L}}$ and $\mathbf{u}^{\mathrm{H}}$ are related via their respective reconstructions, $\mathbb{B}^{\mathrm{L}}$ and $\mathbb{B}^{\mathrm{H}}$, through \eqref{eq:reconstruction_diff}. Hence,
\begin{equation}
    \eta_T \le C\Delta_t \left\lvert T_y \right\rvert \left\lvert \partial T_x \right\rvert(\Delta \boldsymbol{\xi})^{\mathrm{L}+1}.
\end{equation}
\end{proof}
\end{proposition}

\subsection{Example: Transport Equation}\label{sec:ex:transport}
To illustrate the convergence results of Thm. \ref{thm:pdf_error_convergence} and Thm. \ref{thm:CDF_convergence_SFV}, consider the transport equation of \eqref{eq:stochastic_cons_law_param} with deterministic flux \eqref{eq:advection} and uncertain initial conditions, given by
\begin{align}\label{eq:ex:example_advection_problem}
&{u}_t + a u_x = 0, \quad {x}\in\Omega_{\mathrm{phys}}\equiv\left[0, \,1 \right],\, {y}\in D_{\mathrm{stoch}}\equiv\left[ 0,\,1 \right],\\ 
 &{u}({x},0,{y}) = \sin(4\pi x)\sin(4\pi y), \quad \mathbf{x}\in \Omega_{\mathrm{phys}},\; \mathbf{y}\in D_{\mathrm{stoch}},
\end{align}
under periodic boundary conditions, with $a=1$. According to the \textit{method of characteristics} \cite{toro2009}, the solution $u$ is
\begin{equation}\label{eq:ex:solution}
    u(x,\,t,\,y) = \sin\left(4\pi(x - at) \right)\sin(4\pi y),
\end{equation}
which we employ to examine the approximate quantities fielded by the SFV scheme applied to high order finite volume schemes in the spatial variables. Specifically, the discretization leverages a fifth-order WENO scheme, i.e., $r=\ell=5$.

 \begin{figure}[t]
        \centering
       \begin{subfigure}[b]{0.48\linewidth}{\begin{center}\includegraphics{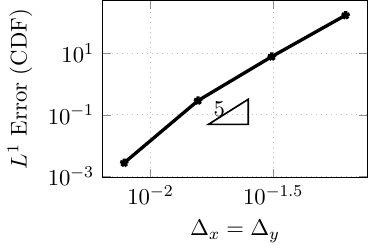}
\end{center}}
\caption{}
\end{subfigure}
       \begin{subfigure}[b]{0.48\linewidth}{\begin{center}\includegraphics{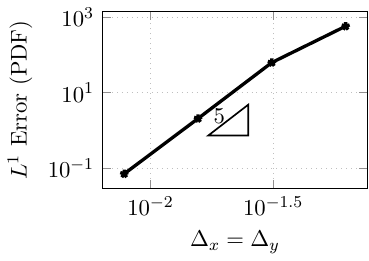}
\end{center}}
        \caption{}
\end{subfigure}
\caption{Numerical confirmation of the \textit{a priori} convergence rates of Thms \ref{thm:pdf_error_convergence} and \ref{thm:CDF_convergence_SFV} for the transport problem in Section \ref{sec:ex:transport}. (a) CDF. (b) PDF. Neither quantities are normalized.}
        \label{fig:ex:advection_convergence}
\end{figure}

As shown in Fig. \ref{fig:ex:advection_convergence}, we obtain the convergence rates implied by Thms. \ref{thm:pdf_error_convergence} and \ref{thm:CDF_convergence_SFV}, resulting in a a quintic relationship between the $L^1$ errors and the discretization size $\Delta_x$ = $\Delta_y$.

\subsection{Refinement}\label{sec:refinement}
In the conventional adaptivity workflow
$$\boxed{\mathrm{SOLVE}}\rightarrow\boxed{\mathrm{ESTIMATE}}\rightarrow\boxed{\mathrm{MARK}}\rightarrow\boxed{\mathrm{REFINE}}, $$
the indicators $\eta_i$ constitute the $\boxed{\mathrm{ESTIMATE}}$ stage, serving as the ingredient to the $\boxed{\mathrm{MARK}}$ process. As further discussed in Section \ref{sec:equilibration}, any marking scheme selects a subset $\mathcal{M}\subset\mathcal{T}$ according to a variety of criteria, e.g., bulk fractions, global tolerances, or local tolerances. In many instances, $\mathcal{M}$ is augmented by a small, auxiliary set of cells that would otherwise not require refinement in order to satisfy the constraints outlined in Section \ref{sec:refinement_framework}.

Implicitly, this $\boxed{\mathrm{MARK}}$ step includes an underlying decision process that determines \textit{how} to refine, in addition to \textit{what} to refine. Fundamentally, the choices available to drive error control in the physical and stochastic spaces rests upon the fidelity and efficiency of the function space $\mathbb{V}$. Enrichment, for example, of the local function spaces in a DG ansatz (so-called $p$-refinement), and enrichment of the spatial resolution ($h$-refinement) constitute the two primary approaches for improving discretization effectiveness. Here we aim to study $h$-refinement modalities under the piecewise constant ansatz of the preceding sections. Constrained to isotropic refinements, error indicators, which ideally link some estimate of discretization error to the actual quality of output functionals and quantities, e.g., the expectation or variance of some component of the solution, provide sufficient decision information. In other words, under simpler models of adaptivity, knowledge of \textit{what} to refine implies \textit{how} to refine.

On the other hand, scalability to higher-dimensions, and demands for greater computational efficiency, motivate more sophisticated refinement structures. Anisotropy in refinements, with adaptivity not only of where to refine but also in the direction of refinements, facilitates a more focused insertion of new DoFs. In multiple dimensions, where certain parameters in a given space-time control volume exhibit varying levels of influence on the approximation quality, isotropic refinements generate impractically large numbers of new DoFs. For a $d\times p$-dimensional computational space, a single isotropic refinement generates $2^{d\times p}$ new control volumes, resulting in a net change of at least $2^{d\times q} - 1$ DoFs. With full anisotropy, we may induce a net change as low as a single DoF regardless of the total dimension of the computational space.

Even so, isolated refinement indicators that provide only some metric of importance or local contributions to the error do not provide a basis for anisotropy decisions. Instead, we seek auxiliary methods of categorization, employing the notions of directionality and smoothness to determine---when \textit{what} to refine is insufficiently informative---in which directions we need to insert new DoFs. Compelling choices include decay rate estimation of Fourier or Legendre expansions \cite{fehling2020, harmon2022}, and directionality of cell boundary error contribution estimates \cite{corrado2022}. We opt to instead consider the following approach.

Let $\mathbb{B}^{1\mathrm{D}}_{U(\mathcal{T})}(T;\, \xi,\, \mathcal{O}): \mathbb{V}(\mathcal{T})\rightarrow \mathbb{R}$ denote a one-dimensional reconstruction along the $\mathcal{O}$-coordinate direction. Suppose that $\mathbb{B}^{1\mathrm{D}}_{U(\mathcal{T})}$ is an $\ell$th degree reconstruction. We estimate from $\mathbb{B}^{1\mathrm{D}}_{U(\mathcal{T})}$ the smoothness in each direction via
\begin{equation}\label{eq:smoothness_indicator}
    \beta_{\mathcal{O}} = \sum_{i = 1}^{\ell} \int_{T_{\mathcal{O}}} (\Delta \xi)^{2i - 1} \left(  \frac{\partial^i \mathbb{B}^{1\mathrm{D}}_{U(\mathcal{T})}(T;\, \xi,\, \mathcal{O})}{\partial^i \xi}  \right)^2\, d\xi,
\end{equation}
where $T_{\mathcal{O}}$ denotes a one-dimensional slice of $T$ in the $\mathcal{O}$-direction. See \cite{jiang1996} for a rationalization for such a smoothness indication structure. Naturally, different realizations of $\mathbb{B}^{1\mathrm{D}}_{U(\mathcal{T})}$ will yield different smoothness indications. Taking a quadratic ENO central stencil interpretation, for instance, generates 
\begin{equation}
    \beta_{\mathcal{O}} = \frac{13}{12}\left(U\left(T_{\mathcal{O}-1}\right) - 2U\left(T\right) + U\left(T_{\mathcal{O}+1}\right)\right)^2 + \frac14\left(U\left(T_{\mathcal{O}-1}\right) - U\left(T_{\mathcal{O}+1}\right)\right)^2,
\end{equation}
where $T_{\mathcal{O}-1}$ and $T_{\mathcal{O}+1}$ are the left and right neighbors of $T$ in the $\mathcal{O}$-direction, respectively. This interpretation matches the perspective, if not the goal, in selecting the smoothest stencil in ENO reconstruction or producing a convex combination of ENO stencils as in WENO. Regardless of the exact form of the smoothness estimator, we assume $\beta_{\mathcal{O}}$ is non-negative, and, without loss of generality, that large values imply local non-smoothness and therefore indicate the need to allocate more DoFs in that direction.

Under such conditions, we propose a bulk criterion for the anisotropy decision. Namely, when 
\begin{equation}\label{eq:anisotropy_decision} 
\beta_{\mathcal{O}} > \varepsilon_{\mathrm{aniso}} \sum^{d+q}_{i=1} \beta_{i},\quad \varepsilon_{\mathrm{aniso}} < \frac{\max_i{\beta_i}}{\sum^{d+q}_{i=1} \beta_{i}} 
\end{equation} 
for any permissible $\mathcal{O}\in \left\{1,\dots,\,d+q\right\}$ (physical or stochastic), we perform a bisection in that direction. Values of $\varepsilon_{\mathrm{aniso}}$ near zero (and below) tend toward isotropic refinements, while values near the upper-bound heavily prefer anisotropic refinements. Depending on the choice of $\varepsilon_{\mathrm{aniso}},$ therefore, dominant directions of non-smoothness will see greater expansions in the number of DoFs compared with directions of smooth variation. Note that, as in the expansion of $\mathcal{M}$ to satisfy the constraints of Section \ref{sec:refinement_framework}, interactions between neighboring cells may also drive expanded subdivisions in other directions as necessary.

After insertion of new DoFs, we perform a conservative interpolation as described in Def. \ref{def:cons_interp}, to assign updated values. We employ the interpolant \eqref{eq:interpolant} in the numerical examples of Section \ref{sec:numerical_results}.

\subsection{Coarsening}\label{sec:coarsening}
    Though conceptually inverse to the problem of refinement, effective coarsening poses significant additional methodological challenges. Whereas refinement attempts to drive error reduction, coarsening strives to increase computational efficiency without deteriorating approximation quality. As an illustration of this dichotomy, consider the tracking of a shock front during the evolution of an underlying PDE. Where refinements capture the behavior of the shocked flows, coarsening enables an unwinding of high resolution when the non-smooth, difficult to capture flows have entered other regions or have left the computational domain. Retaining DoFs in such cases burdens the spatial and temporal discretization procedures without benefit to the quality of output functionals.
    
    Recall Fig. \ref{fig:example_discretizations} and the corresponding constraints on the set of permissible discretizations $\mathbb{T}$ outlined in Section \ref{sec:refinement_framework}. When executing refinements according to the adaptivity workflow in Section \ref{sec:refinement}, ensuring that constraints are satisfied will in many cases require additional refinements of neighboring cells when executing an instruction that would otherwise generate some $\mathcal{T}\notin \mathbb{T}$. Like refinement, coarsening engages a semi-local procedure by scanning neighbors. Importantly, however, coarsening considers the \textit{future} state of the neighboring cells rather than the current state as in the case of refinement. 
    
     Due to the hierarchical, forest mesh structure of our discretization framework, coarsening here is \textit{exclusively} an unwinding procedure, meaning that only arrangements that were possible with the existing forest are permitted. Cells that are not siblings may not merge together as this would disrupt the trees spawned from the root cells, complicating conservation and mapping between function spaces.
     
     Assuming, then, that a cell $T\in\mathcal{T}$ is marked for coarsening based on the supplied criteria,
     proceeding with coarsening requires first checking that its sibling cells also request coarsening. Afterwards, compatibility checks with the future states of any other $T'\in\mathcal{T}$ such that $f_y\in\mathcal{F}_y(T)$ and $f'_y\in\mathcal{F}_y(T')$ intersect with non-zero measure will either cancel the decision to coarsen or allow coarsening to proceed. Given the possibility of propagation in these cancellation instructions, the procedure is necessarily iterative.
     
     Now, as a final step in executing coarsening instructions, we map the data on the refined siblings to their immediate parent via a conservative projection as described in Def. \ref{def:cons_projection}, reducing the local resolution as desired without upsetting expected solution properties.
     
     We note that coarsening, overall, aims to encourage \textit{equilibration}, a condition of optimality for numerical discretizations \cite{chen1994}. We further discuss this property and its role in designing adaptive schemes in Section \ref{sec:equilibration}.
    
\subsection{Equilibration}
\label{sec:equilibration}
Conceptually, equilibration implies that no region of the discretization is significantly higher or lower in effective resolution than any other region. If such a situation occurs, reallocation of DoFs from an inefficient region (over-refined) to an insufficient region (under-refined) will yield an enhancement to efficiency. Depending on the goal of a simulation, e.g., computing the expectation at a single point in space, or some global quantity, the quantification of ``effective resolution" will change.

The promotion of this property heavily influences the adaptivity scheme. Take, for example, the conventional Dörfler marking \cite{dorfler1996}:

\textit{Determine a set $\mathcal{M}\subset{\mathcal{T}}$ of minimal cardinality such that}
\begin{equation}
 \theta\sum_{T\in\mathcal{T}} \eta_T^2 \le \sum_{T\in\mathcal{M}} \eta_T^2,\quad 0 < \theta < 1.
\end{equation}
Dörfler marking selects a fraction $\theta$ of cells for refinement and applies an iterative framework to drive numerical error down in an order implied by the rankings of $\eta_T$. Even so, if we assume equality of each $\eta_T$, the Dörfler strategy immediately disrupts equilibration.

Instead, suppose that an oracle has determined a fixed local tolerance $\mathrm{TOL}$ that achieves a desired error quantity. Based on this local tolerance, in each iteration we drive each $\eta_T$ towards $\mathrm{TOL}$, unconditionally when $\eta_T > \mathrm{TOL}$, and if $\eta_T < \mathrm{TOL}$, then only when predicted that the tolerance will still be met after coarsening. 

Given the deleterious impacts of over-coarsening as opposed to over-refinement, we perform coarsening only when $\eta_T < \theta \mathrm{TOL},$ for $0 < \theta < 1$ typically in the neighborhood of $\theta=0.1$, i.e., that the refinement indicator is a tenth of the local threshold.
\subsection{The Unsteady Enriched-Reduced $\mathbb{B}$ Predictor-Corrector Scheme}\label{sec:predictor_corrector}
Encapsulating the contributions of Sections \ref{sec:convergence} through \ref{sec:equilibration}, we present a predictor-corrector adaptive method based on an enriched-reduced $\mathbb{B}$ pair for unsteady problems.
 \begin{algorithm}
\caption{Unsteady Enriched-Reduced $\mathbb{B}$ Adaptivity}
\label{alg:adaptivity}
\begin{algorithmic}[1]
\STATE{Define \{$\varepsilon, \, \varepsilon_{\mathrm{aniso}},\, t_n,\, \mathbb{B}^{\mathrm{H}},\, \mathbb{B}^{\mathrm{L}},\, \mathbb{K},\,\theta,\, \mathcal{T}$\}}
\STATE{${\eta} := \infty$}
\WHILE{${\eta} > \varepsilon$}
\STATE{${\eta} = 0$}
\FORALL{$T\in\mathcal{T}$}
\STATE\label{line3}{Perform reconstruction according to \eqref{eq:reconstruction_order} and flux integration for an enriched-reduced reconstruction pair $\mathbb{B}^{\mathrm{H}}$ and $\mathbb{B}^{\mathrm{L}}$ of order $\mathrm{H}$ and $\mathrm{L}$, respectively}
\STATE{Compute $\mathbb{K}(\mathbf{U}^{\mathrm{H}}_n(T)) - \mathbb{K}(\mathbf{U}^{\mathrm{L}}_n(T))$ via $\mathbb{K}$ as in \eqref{eq:partial_step_error} and Prop. \ref{prop:error_bound}}
\STATE{Determine $\eta_T$ from \eqref{eq:computable_refinement_indicator}}
\IF{$\eta_T > \varepsilon$}
\STATE{Compute smoothness indicators $\eta_i^{\mathrm{aniso}}$ via \eqref{eq:smoothness_indicator}}
\STATE{Assign refinement directions via \eqref{eq:anisotropy_decision} with $\varepsilon_{\mathrm{aniso}}$}
\STATE{Stage refinement instruction for $T$}
\ELSIF{$\eta_T < \theta \varepsilon$}
\STATE{Stage coarsening instruction for $T$ to its immediate ancestor}
\ENDIF
\STATE{${\eta} = \max{\left( \eta,\, \eta_T \right)}$}
\ENDFOR
\STATE{Amend coarsening and refinement instructions to satisfy the constraints of Section \ref{sec:refinement_framework}}
\STATE{Execute instructions, updating $\mathcal{T}$}
\ENDWHILE
\STATE{$t_{n+1}\leftarrow t_{n}$}
\RETURN $\max_{T\in\mathcal{T}}{\eta_T}$
\end{algorithmic}
\end{algorithm}

\begin{remark}
Algorithm \ref{alg:adaptivity} includes localized refinements in the combined physical and stochastic computational space \textit{and} attribution in particular directions. For the conventional case of isotropic refinements, however localized to regions of the physical and stochastic spaces, letting $\varepsilon_{\mathrm{aniso}} \le 0$ suffices.
\end{remark}

For a particular realization of Algorithm \ref{alg:adaptivity}, we take $\mathbb{B}^{\mathrm{H}}$ as a piecewise-quadratic fifth-order accurate WENO reconstruction, $\mathbb{B}^{\mathrm{L}}$ as a piecewise-linear third-order accurate WENO reconstruction, and $\mathbb{K}$ denotes a first-order Euler time integration step. The coarsening bulk factor $\theta$ is, at least heuristically, related to efficiency and equilibration goals as discussed in Section \ref{sec:coarsening} and Section \ref{sec:equilibration}. In the Numerical Results section, we exclusively employ $\theta = 0.1$. The error tolerances $\varepsilon$ and $\varepsilon_{\mathrm{aniso}}$ are chosen according to accuracy requirements, with smaller values of both driving greater allocations of DoFs.

\begin{remark}
The choice of $\mathbb{K}$ is potentially disjoint from the time integrator employed to update the solution at $t_n$ to that at $t_{n+1}$. In other words, choosing a first-order integrator for refinement indication does not preclude a higher-order update for the solution. In the following section, we employ a third-order SSP-RK integrator for this update. The first-order time-integrator $\mathbb{K}$ is extracted from the higher-order time integrator, which is performed only partially until satisfaction of the refinement criteria. 
\end{remark}

\section{Numerical Results}
\label{sec:numerical_results}
We apply the proposed scheme to the stochastic Burgers' equation and the stochastic Euler equations subject to random initial conditions. The procedure for modeling other sources of uncertainty, such as in boundary conditions, fluxes, EOS, etc., is unchanged from the following presentation. 

Under the equilibration strategy of Algorithm \ref{alg:adaptivity}, we assign various tolerances $\varepsilon$ for driving adaption of the discretization, with associated anisotropic tolerances $\varepsilon_{\mathrm{aniso}}$. To reduce discretization error in the physical and stochastic spaces, we refine anisotropically according to the directionality of non-smoothness indicated by \eqref{eq:smoothness_indicator}. At each time step $t_n$, we perform error estimation via the enriched-reduced pair according to an embedded explicit Euler step, where satisfaction of the local error tolerance engages the remainder of the high-order SSP scheme to advance the time-step to $t_{n+1}$. Otherwise, the discretization is refined based on the local error indicators, and a new time step is performed starting from the state at $t_n$ mapped to the new discretization. We summarized the details of this procedure in Algorithm \ref{alg:adaptivity}.

 \begin{table}[htb]
 \centering
 \begin{tabular}{c||c|c}
      Index & $\varepsilon/\Delta_t$ & $\varepsilon_{\mathrm{aniso}}$ \\
      \hline
      0 & $5\times 10^{-4}$ & 0.5 \\
      1 & $2\times 10^{-4}$ & 0.5 \\
      2 & $5\times 10^{-5}$ & 0.5 \\
      3 & $1\times 10^{-5}$ & 0.5 \\
 \end{tabular}
 \caption{The four tolerance-pairs tested for the Burgers' equation problem of Section \ref{sec:burger_1D}.}
 \label{tab:tolerances}
\end{table}

To quantify this performance, and the enhanced convergence rates afforded by the proposed adaptivity, we perform a series of tests for the choices in Table \ref{tab:tolerances}. Typical computations in practical applications, e.g., via MCS, prioritize low order stochastic moments. In the SFV framework, these quantities amount to integration of piecewise polynomials weighted by the probability density functions of the underlying uncertainty. Secondly, with robust approximations of push-forward probability densities, we may recover additional statistics disjoint from the SFV approximation. For this, we follow the procedure in Thm. \ref{thm:pdf_error_convergence}, which amounts to sampling a piecewise polynomial response. We evaluate the convergence of these quantities globally in the $L^1$ sense by comparing the performance with respect to a high-resolution reference computation.  Specifically, the errors in expectation, variance, and distribution are evaluated as
\begin{align}\label{eq:l1_numerical_result_postproc}
    \left\lVert \mathbb{E}[\bar{u}] - \mathbb{E}[u] \right\rVert_{L^1} &= \int_{\Omega_{\mathrm{phys}}} \left\lvert \mathbb{E}[\bar{u}(x)] - \mathbb{E}[u(x)] \right\rvert\, dx,\\
    \left\lVert \mathrm{Var}[\bar{u}] - \mathrm{Var}[u] \right\rVert_{L^1} &= \int_{\Omega_{\mathrm{phys}}} \left\lvert \mathrm{Var}[\bar{u}(x)] - \mathrm{Var}[u(x)] \right\rvert\, dx,\\
    \label{eq:l1_numerical_result_postproc_3}
    \left\lVert \mathrm{PDF}[\bar{u}] - \mathrm{PDF}[u] \right\rVert_{L^1} &= \int_{\Omega_{\mathrm{phys}}} \int^{\infty}_{-\infty} \left\lvert \mathrm{PDF}[\bar{u}(x)](v) - \mathrm{PDF}[u(x)](v) \right\rvert \, dx dv.
\end{align}
For evaluating $\eqref{eq:l1_numerical_result_postproc}$, as in proof of Thm. \ref{thm:CDF_convergence_SFV}, we identify an index set $\mathcal{I}(x)$ of cells $T\in\mathcal{T}$ for a finite set $\left\{ \mathcal{X} \right\}$ of $x\in\Omega_{\mathrm{phys}}$ matching quadrature points over the subdivided physical space. Integration along the stochastic directions is performed on a cell-by-cell basis. Finally, while certain regions may attain high-order accuracy locally, the global quantities include multiple shocks, limiting the expected optimal global convergence rate to first-order.

\subsection{Burgers' equation with initial value uncertainty}\label{sec:burger_1D}
    Consider the Burgers' flux of \eqref{eq:burgers_flux} and the stochastic PDE of \eqref{eq:stochastic_cons_law} with the uncertainty parameterized by a random variable $\mathbf{y}$ as in \eqref{eq:stochastic_cons_law_param}. We construct a coarse initial discretization of the computational space $\Omega_{\mathrm{phys}}\times D_{\mathrm{stoch}}$ via control volumes under a piecewise constant ansatz. In each dimension (physical and stochastic), we initialize the mesh with 16 cells, for a total of 256 DoFs. For reconstruction in the physical and stochastic spaces, we employ a piecewise quadratic and piecewise linear enriched-reduced WENO pair, which under conditions of sufficient smoothness, offer fifth-order and third-order accuracy, respectively. A low-storage explicit third-order SSP-RK method serves for time integration \cite{ketcheson2008}.

    As a particular benchmark, we study the example introduced in a recent study \cite{herty2023}, taking
 \begin{equation}\label{eq:herty_ic} 
 \mathbf{u}(\mathbf{x},0,\mathbf{y}) = \sin\left( 2\pi\mathbf{x} \right)\sin\left( 2\pi\mathbf{y}\right), \; \mathbf{x}\in[0,\,1],\quad \mathbf{y}\in D_{\mathrm{stoch}}\subset\mathbb{R},
 \end{equation}
 assuming that a sufficiently accurate bounded $D_{\mathrm{stoch}}$ can approximate uncertainties of infinite support, as described in Remark \ref{remark:unbounded}. We assume periodic boundary conditions and a deterministic flux. Though globally smooth, two stationary shocks form: one at $\mathbf{x} = 0.5, \, \mathbf{y} < 0.5$, and the other at $\mathbf{x}=0,\,\mathbf{x}=1,\, \mathbf{y} > 0.5 $, plus periodicity in $\mathbf{x}$ and $\mathbf{y}$. As in the previous study \cite{herty2023}, we take the terminal time as $t=0.35$ to permit formation of these two stationary shocks.

\begin{figure}[!htb]
        \centering
       \begin{subfigure}[b]{0.40\linewidth}
       {
       \begin{center}
        \hspace{0.75cm}
        \includegraphics{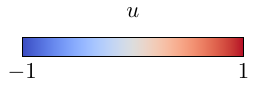}\\
\end{center}
\begin{center}
      \includegraphics{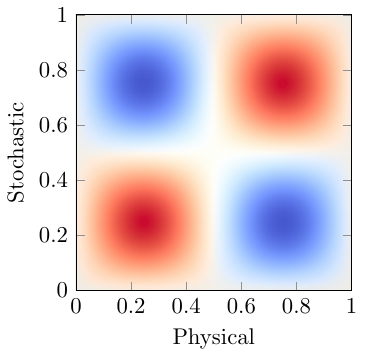}
\end{center}
}
        \caption{}
        \label{fig:sub:burger_initial}
\end{subfigure}
\begin{subfigure}[b]{0.40\linewidth}{
\begin{center}
    \hspace{0.75cm}
             \includegraphics{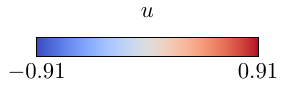}\\
\end{center} 
\begin{center}
       \includegraphics{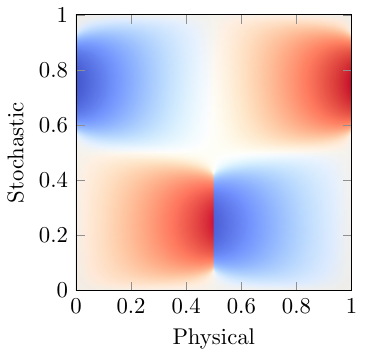}
\end{center}}
        \caption{}
        \label{fig:sub:burger_final}
\end{subfigure}
\caption{Initial and terminal states for the Burgers' equation problem of Section \ref{sec:burger_1D} with uncertain initial data. (a) The initial state for the problem, i.e., the initial condition \eqref{eq:herty_ic}. (b) The terminal state at $t=0.35$.}
        \label{fig:burgers_initial_and_final_state}
\end{figure}

 The solution at $t=0.35$, depicted in Fig. \ref{fig:sub:burger_final}, illustrates the formation of these two shocks from smooth initial conditions in Fig. \ref{fig:sub:burger_initial}. Note that, in this example, traversal of the stochastic space does not cross a shock boundary. We explore such cases in Section \ref{sec:euler_1D}.
 
 We further suppose an oracle supplies an expression of the random variable $\mathbf{y}$, namely $\mathbf{y}_1\sim \mathcal{B}(2,5).$ While the expression of the uncertainty impacts the \textit{results} of adaptivity, the \textit{procedure} is entirely agnostic to this choice. The reference solution is obtained via fine tolerance adaptive simulation with resolutions equivalent to uniform discretizations of approximately 70 million DoFs.
 
 Starting with a uniform discretization of $16$ cells in each dimension, we consider the four pairs of tolerances $\left\{ \varepsilon,\, \varepsilon_{\mathrm{aniso}}\right\}$ indicated in Table \ref{tab:tolerances}.

The remaining instruments of Algorithm \ref{alg:adaptivity} are as described in Section \ref{sec:predictor_corrector}. Note that the listed tolerances are normalized with respect to the temporal discretization, leading to the interpretation of $\varepsilon$ as the space-time control volume error contribution tolerance. For numerical integration, we employ Gauss-Lobatto quadrature rules of sufficiently high order to preserve asymptotic characteristics. Finally, the necessary WENO weights for evaluating reconstructions $\mathbb{B}^{\mathrm{H}}$ and $\mathbb{B}^{\mathrm{L}}$ may be pre-computed for reference cells analogously to the procedure previously outlined for Gauss-Legendre rules \cite{titarev2004}.

With the developed SFV scheme, we may easily query statistical moments, point values, and other quantities and functionals of the solution to the stochastic IBVP parameterized by random variables. For instance, we illustrate in Fig. \ref{fig:burgers_1D_solution_plots} the solution behavior as described by the first two moments. Alongside depictions of the perturbations about the mean with the computed standard deviation, we include a confidence region defined by the first and third quartiles. Note that the probability density of $u$ at each position $x$ is not Gaussian. In particular, the results induced by the choice of $\mathbf{y}_1\sim\mathcal{B}(2,\,5)$ in Fig. \ref{fig:burgers_1D_solution_plots}, reflect a significant bias in the output solution. As a result of the underlying distribution's emphasis on the central shocked region of the computational domain, we see much larger expected values in the neighborhood of $x = 0.5$ compared to what would occur for more uniform distributions. Furthermore, the quartiles computed from the SFV approximation underscore the sensitivity of the expectation to extreme values and the merits of more complete pictures of stochasticity.

\begin{figure}[!tb]
    \begin{center}
        \hspace{1cm}
        \includegraphics{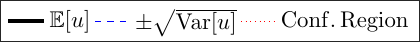}\\
    \end{center}
    \begin{center}
\includegraphics{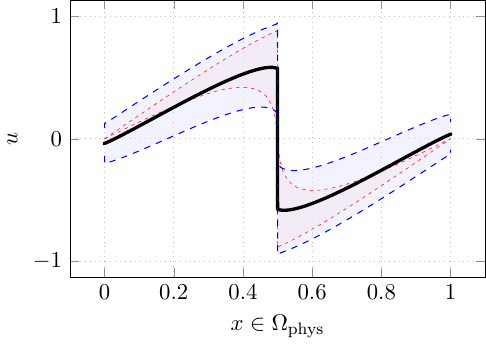}
    \end{center}
\caption{Solution characteristics of $u$ at the terminal time $t=0.35$ for the problem of Section \ref{sec:burger_1D} as described by the first stochastic moments. The red shaded region denotes the first-to-third quartile confidence region.}
        \label{fig:burgers_1D_solution_plots}
\end{figure}

According to the adaptive scheme outlined in Section \ref{sec:error_analysis}, this parametric highlighting of the stochastic space via $\mathbf{y}_1$ translates to tailored discretizations, which respond to the evolution of the PDE and the underlying stochasticity upon which the parameterization is based. Fig. \ref{fig:burgers_1D_discretizations} illustrates this characteristic. Note that shocked flows in lower probability regions, while not neglected, see fewer refinements compared to high probability flows, resulting in a concentration of DoFs. Matching the resolution attained with adaptivity would require millions of DoFs with uniform (i.e., global) refinements. Additionally, the anisotropic character of the refinements permits targeting the directionality of solution behavior, including the discontinuities located in fixed spatial coordinates.
\begin{figure}[!htb]
        \centering
       \begin{subfigure}[b]{0.35\linewidth}{
       \begin{center}
       \includegraphics{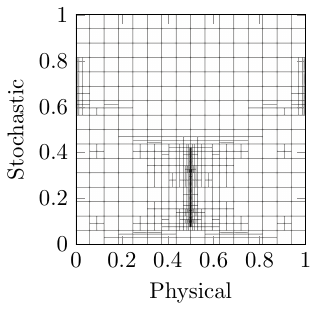}
\end{center}
       }
        \caption{}
        \label{fig:sub:burgers_discretization_beta_coarse}
\end{subfigure}
\begin{subfigure}[b]{0.35\linewidth}{
       \begin{center}\includegraphics{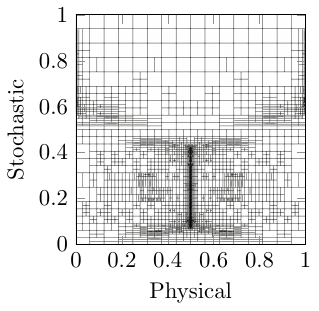}
\end{center}
}
        \caption{}
        \label{fig:sub:burgers_discretization_beta_fine}
\end{subfigure}
\caption{Refined discretizations at the terminal time $t=0.35$ for the problem of Section \ref{sec:burger_1D} with $\mathbf{y}_1\sim\mathcal{B}(2,5)$. (a) The discretizaton for the second coarsest tolerance pair in Table \ref{tab:tolerances}. (b) The discretization for the finest tolerance in Table \ref{tab:tolerances}.}
        \label{fig:burgers_1D_discretizations}
\end{figure}

 \begin{figure}[!htb]
    \centering
    \hspace{1cm}
        \includegraphics{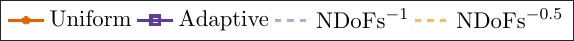}
        \\
        \begin{subfigure}[b]{0.325\linewidth}{\begin{center}\includegraphics{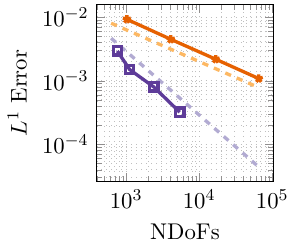}\end{center}
}
\caption{$\mathbb{E}[u] \mid \mathbf{y}_1$}
\end{subfigure}
       \begin{subfigure}[b]{0.325\linewidth}{\begin{center}\includegraphics{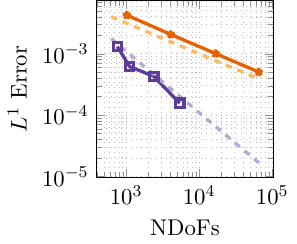}\end{center}
}
        \caption{$\mathrm{Var}[u] \mid \mathbf{y}_1$}
\end{subfigure}
       \begin{subfigure}[b]{0.325\linewidth}{\begin{center}\includegraphics{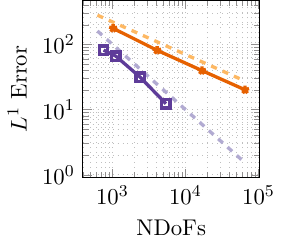}\end{center}
}
\caption{$\mathrm{PDF}[u] \mid \mathbf{y}_1$}
\end{subfigure}
\caption{$L^1$ error of the first two stochastic moments and the probability density function of $u$ at time $t=0.35$ for uniform refinements and the proposed adaptivity with $\mathbf{y}_1\sim\mathcal{B}(2,\,5)$. Dashed lines indicate convergence rates with respect to the number of degrees of freedom. Note that these quantities are not normalized.}
        \label{fig:burgers_1D_rv_uniform_convergence}
\end{figure}

Recall the three quantities of interest defined in \eqref{eq:l1_numerical_result_postproc}-\eqref{eq:l1_numerical_result_postproc_3}. Considering these objectives in sequence, we compare in Fig. \ref{fig:burgers_1D_rv_uniform_convergence} the proposed adaptivity and uniform, global refinements, which do not consider localized relationships between the physical and stochastic coordinates. Where in all tests, uniform refinement attains empirical convergence rates related to the square root of the NDoFs, the adaptive method drives linear convergence. Examining the $L^1$ norms for the expectation and the variance over the tested tolerances, we observe a 1 to 2 orders of magnitude improvement for the same NDoFs. In the case of the $L^1$ error norm for the push-forward probability density in Fig. \ref{fig:sub:burgers_1d_unif_pdf}, the enhanced convergence rate is maintained for the adaptive method. While each tolerance adapts the discretization independently, our proposed method yields highly consistent convergence behavior.

 \subsection{Euler equations with initial value uncertainty}\label{sec:euler_1D}
 Consider now Euler's flux of \eqref{eq:euler_flux}. We take the same predicted-corrector, time integrator, and tolerances as deployed in Section \ref{sec:burger_1D}.
 
 We discretize $\Omega_{\mathrm{phys}} \equiv \left[ 0,\,1 \right]$ with free flow boundary conditions, and set the temporal region of interest as $\mathbb{I}\equiv \left[0,\,0.1\right]$. The starting discretization consists of 24 cells in each dimension, or 576 DoFs.
 
 We suppose an ideal gas of internal energy
 \begin{equation}
     e = \frac{p}{(\gamma - 1)\rho},
 \end{equation}
 for a ratio of specific heats $\gamma > 1$. In particular, we take a deterministic value of $\gamma = 1.4$.
 We suppose that an oracle provides uncertain initial data of the form
 \begin{equation}\label{eq:euler_1D_ini_con}
     \mathbf{u}(\mathbf{x},\,0,\,\mathbf{y}) = \begin{cases}\left( 1,\, 0,\, 0.5 + 2.5\mathbf{y} \right), & \mathbf{x} < 0.5\\ \left( 0.125,\, 0,\, 0.25 \right), & 0.5 < \mathbf{x} < 0.75\\ \left( 0.5,\, 0,\, 0.25 + 1.25\mathbf{y}\right), &\mathbf{x} > 0.75\end{cases}, \;\mathbf{x}\in[0,\,1],\, \mathbf{y}\in D_{\mathrm{stoch}}\subset\mathbb{R},
 \end{equation}
 which is equivalent to uncertain initial pressures.   The initial condition \eqref{eq:euler_1D_ini_con} is similar to a stochastic three-state Woodward-Colella problem \cite{woodward1984}, with the terminal state of the parametric problem shown in Fig. \ref{fig:euler_1D_final_state} at $t=0.1$. As the simulation evolves, discontinuous flows collide and interact in the stochastic and physical spaces, leading to highly varied solution behavior.

  \begin{figure}[!htb]
        \centering
     \begin{subfigure}[b]{0.32\linewidth}{\begin{center}\includegraphics{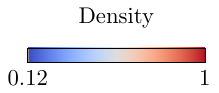}\\
\end{center}
\begin{center}
       \hspace{-0.75cm}\includegraphics{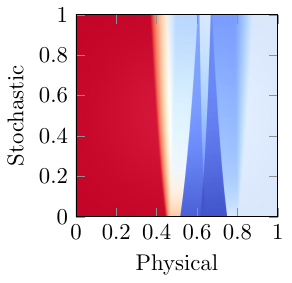}
\end{center}
      }
      \caption{$\rho$}
      \end{subfigure}
      \hfill
      \begin{subfigure}[b]{0.32\linewidth}{\begin{center}\includegraphics{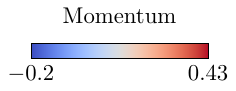}\\
\end{center}
\begin{center}\hspace{-1cm}\includegraphics{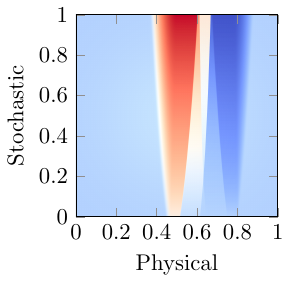}
\end{center}
}\caption{$\rho v$}\end{subfigure}
      \hfill
      \begin{subfigure}[b]{0.32\linewidth}{\begin{center}\includegraphics{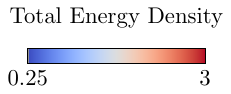}\\
\end{center}
      \begin{center} \hspace{-1cm}\includegraphics{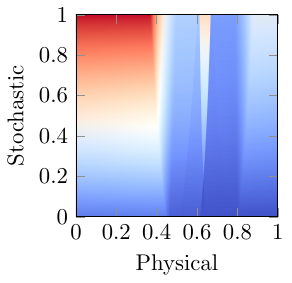}
\end{center}
}
\caption{$\rho E$}\end{subfigure}
        \caption{Terminal state at $t=0.10$ for the Euler's equation problem of Section \ref{sec:euler_1D} with uncertain initial pressures given by the initial condition \eqref{eq:euler_1D_ini_con}. (a) Density. (b) Momentum. (c) Total energy density.}
        \label{fig:euler_1D_final_state}
\end{figure}

 \begin{figure}[!htb]
        \centering
       \begin{subfigure}[b]{0.35\linewidth}{
       \begin{center}
       \includegraphics{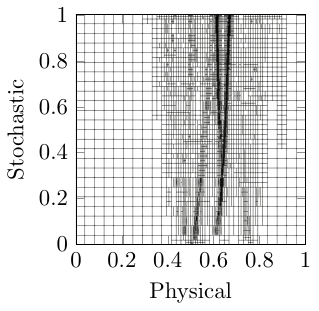}
\end{center}
       }
        \caption{}
        \label{fig:sub:euler_mesh_coarse}
\end{subfigure}
\begin{subfigure}[b]{0.35\linewidth}{
       \begin{center}\includegraphics{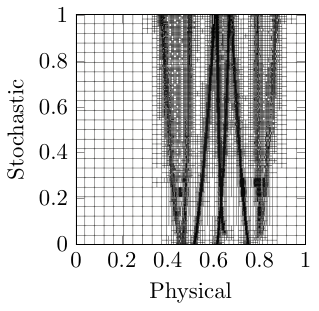}
\end{center}
}
        \caption{}
        \label{fig:sub:euler_mesh_fine}
\end{subfigure}
\caption{Refined discretizations at the terminal time $t=0.1$ for the problem of Section \ref{sec:euler_1D} with $\mathbf{y}_2\sim\mathcal{U}(0,\,1)$. (a) The discretizaton for the second coarsest tolerance pair in Table \ref{tab:tolerances}. (b) The discretization for the finest tolerance in Table \ref{tab:tolerances}.}
        \label{fig:euler_3state_discretizations}
\end{figure}

\begin{figure}[tb]
    \begin{center}
        \hspace{0.5cm}\includegraphics{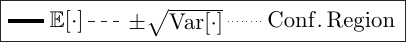}
     \\\end{center}
\centering
\begin{subfigure}[b]{0.32\linewidth}{\includegraphics{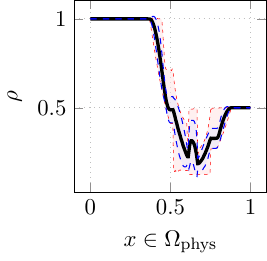}
}
\caption{$\rho \mid \mathbf{y}_2$}
\label{fig:sub:euler_1d_unif_soln}
\end{subfigure}
\hfill
\begin{subfigure}[b]{0.32\linewidth}{\includegraphics{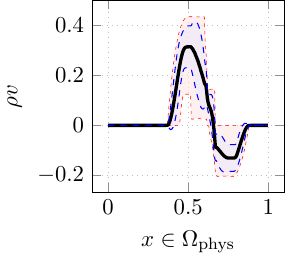}
}
\caption{$\rho v \mid \mathbf{y}_2$}
\label{fig:sub:euler_1d_uni_soln_momentum}
\end{subfigure}
\hfill\begin{subfigure}[b]{0.32\linewidth}{\includegraphics{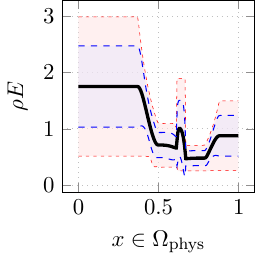}
}
\caption{$\rho E \mid \mathbf{y}_2$}
\label{fig:sub:euler_1d_uni_soln_energy}
\end{subfigure}\\
\caption{Solution characteristics of $\rho, \, \rho v,$ and $\rho E$ at the terminal time $t=0.10$ for the problem of Section \ref{sec:euler_1D} as described by the first stochastic moments. In this instance, the red shaded region denotes the 100\% probability confidence region.}
        \label{fig:euler_1D_solution_plots}
\end{figure}

 Adaptivity proceeds as in the scalar case, with refinement driven by dissatisfaction of the desired local tolerances on each component of $\mathbf{u}$ separately. While large differences in magnitude between the components of the approximate solution $\mathbf{u}$ to \eqref{eq:stochastic_cons_law_param} with flux \eqref{eq:euler_flux} may require setting tolerances relative to each component, we identically employ the tolerances in Table \ref{tab:tolerances} for each component.

 In specifying the initial conditions \eqref{eq:euler_1D_ini_con}, we take $\mathbf{y}$ as $\mathbf{y}_2\sim\mathcal{U}(0,1)$. In this case, the uniformly distributed uncertainty drives full capture of the discontinuous and rarefied flows, posing more challenge than, e.g., random variables following beta or normal distributions that possess low probability flows. As before, different choices of $\mathbf{y}$ will lead to different weighted prioritizations of the approximation quality and probabilities exercised by Algorithm \ref{alg:adaptivity}, yet the essential 
 process remains unchanged.
 
 Considering the solution behavior as before in more detail, we examine in Fig. \ref{fig:euler_1D_solution_plots} the three components of the solution vector as described by the first two statistical moments and a $100\%$ probability confidence region. Note that the first quarter and the last eighth of the domain remains deterministic with respect to the density and momentum due to the terminal simulation time selected. Furthermore, while not invariant to the stochasticity, the relative variation of the density lacks the larger solution diversity exhibited by the momentum and total energy densities. Moreover, for this problem, as demonstrated particularly in Fig. \ref{fig:euler_1D_final_state}, we have discontinuities in the physical \textit{and} stochastic spaces. Yet, the inherent smoothing delivered by the SFV method leads to a corresponding smoothing of the solution profiles in Fig. \ref{fig:euler_1D_solution_plots}, even in regions impacted by colliding flows, such as in the neighborhood of $x=0.6$. The collision of the two flows, however, is not entirely mitigated with this smoothing, as illustrated by the jumps in $\rho$ and $\rho E$.

\begin{figure}[!htb]
    \centering
    \begin{tabular}{cccc}
    \centering
    \rotatebox{90}{\qquad\qquad\quad$\rho$} &
     {\quad\quad\includegraphics{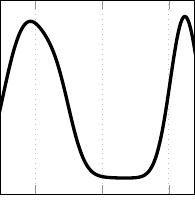}}
    \\
    \rotatebox{90}{\qquad\qquad\quad$\rho v$} &
    {
    \includegraphics{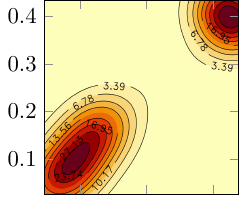}
    }
     &
     {\;\;\includegraphics{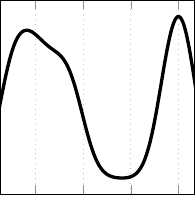}}
    \\
    \rotatebox{90}{\qquad\qquad\quad$\rho E$} &
    {
      \includegraphics{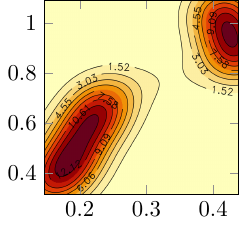}
    } &
    {
      \includegraphics{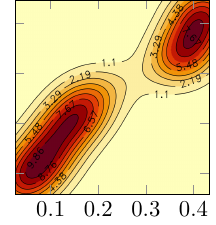}
    } &
     {\;\;\includegraphics{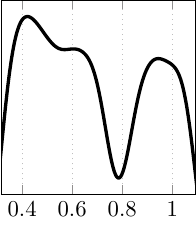}}
    \\
    & \quad$\rho$ & \quad$\rho v$ & \quad$\rho E$
    \end{tabular}
    \caption{Marginal push-forward probability density functions for the Euler problem of Section \ref{sec:euler_1D} computed via the SFV representation. The three components of the solution vector $\mathbf{u} = \left [ \rho,\, \rho v,\, \rho E \right]^\top$ are taken pairwise to compute the associated 1-D and 2-D density functions in the interest of interpretability.}
    \label{fig:euler_1D_marginal_densities}
    \end{figure}
  
 Examining the solution characteristics further, we illustrate in Fig. \ref{fig:euler_1D_marginal_densities} a set of marginal push-forward densities computed via the SFV representation for the solution vector at a point in the physical domain (specifically $x=0.581$) at the terminal time $t=0.1$ that, when considering uncertainty, experiences shocked flows. Note the high concentration of probability to two regions of the solution space in all cases. This property illustrates the effect the shocks propagating in the physical and stochastic spaces induce on the stochastic character of the solution.
 
 \begin{figure}[!htb]%
        \begin{center}
        \hspace{1.5cm}\includegraphics{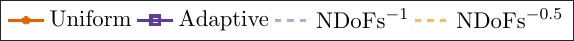}\\
        \end{center}
       \begin{subfigure}[b]{0.29\linewidth}{\begin{center}\includegraphics{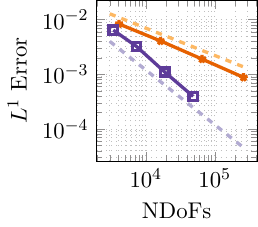}
\end{center}}
\caption{$\mathbb{E}[\rho] \mid \mathbf{y}_2$}
\label{fig:sub:burgers_1d_unif_expectation}
\end{subfigure}
\hfill
       \begin{subfigure}[b]{0.29\linewidth}{\begin{center}\includegraphics{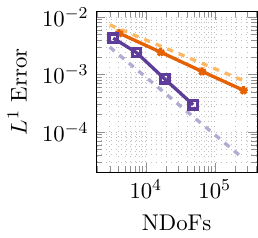}
\end{center}}
        \caption{$\mathbb{E}[\rho v] \mid \mathbf{y}_2$}
        \label{fig:sub:burgers_1d_unif_variance}
\end{subfigure}
\hfill
       \begin{subfigure}[b]{0.29\linewidth}{\begin{center}\includegraphics{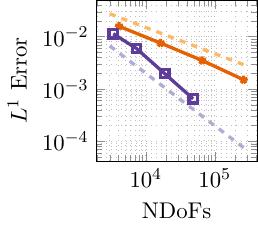}
\end{center}}
\caption{$\mathbb{E}[\rho E] \mid \mathbf{y}_2$}
\label{fig:sub:burgers_1d_unif_pdf}
\end{subfigure}\\
\begin{subfigure}[b]{0.29\linewidth}{\begin{center}\includegraphics{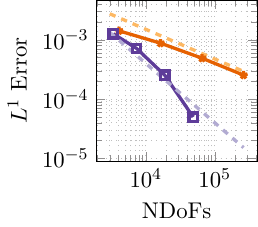}
\end{center}}
\caption{$\mathrm{Var}[\rho] \mid \mathbf{y}_2$}
\end{subfigure}
\hfill
       \begin{subfigure}[b]{0.29\linewidth}{\begin{center}\includegraphics{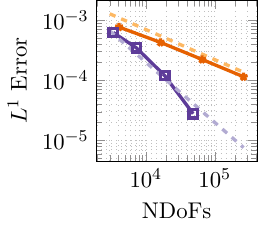}
\end{center}}
        \caption{$\mathrm{Var}[\rho v] \mid \mathbf{y}_2$}
\end{subfigure}
\hfill
       \begin{subfigure}[b]{0.29\linewidth}{\begin{center}\includegraphics{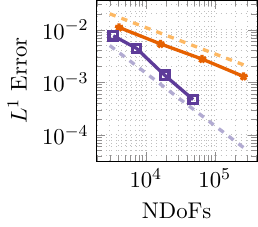}
\end{center}}
\caption{$\mathrm{Var}[\rho E] \mid \mathbf{y}_2$}
\end{subfigure}
\\
\begin{subfigure}[b]{0.29\linewidth}{\begin{center}\includegraphics{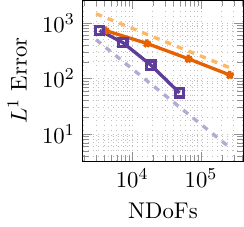}
\end{center}}
\caption{$\mathrm{PDF}[\rho] \mid \mathbf{y}_2$}
\end{subfigure}
\hfill
       \begin{subfigure}[b]{0.29\linewidth}{\begin{center}\includegraphics{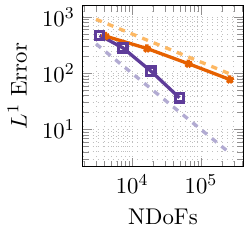}
\end{center}}
        \caption{$\mathrm{PDF}[\rho v] \mid \mathbf{y}_2$}
\end{subfigure}
\hfill
       \begin{subfigure}[b]{0.29\linewidth}{\begin{center}\includegraphics{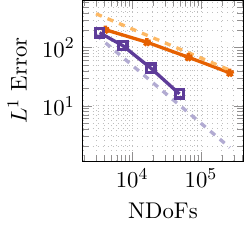}
\end{center}}
\caption{$\mathrm{PDF}[\rho E] \mid \mathbf{y}_2$}
\end{subfigure}
\caption{$L^1$ error of the first two stochastic moments of $\mathbf{u} = \left [ \rho,\, \rho v,\, \rho E \right]^\top$ at time $t=0.10$ for uniform refinements and the proposed adaptivity with $\mathbf{y}_2\sim\mathcal{U}(0,\,1)$. Dashed lines indicate convergence rates with respect to the number of degrees of freedom. Note that these quantities are not normalized.}
        \label{fig:euler_3state_rv_uniform_convergence}
\end{figure}

  Finally, we consider the convergence of the proposed adaptive method against conventional global refinements for each component of the solution vector in Fig. \ref{fig:euler_3state_rv_uniform_convergence}. As in the scalar problem considered in Section \ref{sec:burger_1D}, the proposed adaptive method attains the theoretical optimal linear convergence rate with respect to the number of degrees of freedom (NDoFs) for all quantities, including approximation of the probability densities. Uniform refinements, on the other hand, deliver convergence rates related to $\sqrt{\mathrm{NDoFs}}$. Overall, the economy of DoF allocations permitted by the predictor-corrector scheme, in addition to restoring an optimal convergence rate even in the presence of multiple shocks and rarefied flows, enables an automated framework for studying hyperbolic PDEs under uncertainty with high accuracy.

\section{Conclusions}
\label{sec:conclusions}

We demonstrated an adaptive framework for stochastic hyperbolic PDEs that yields enhanced convergence properties over conventional schemes. Supported by novel theoretical results concerning the approximation properties of valuable computables, in particular push-forward densities, the proposed method supports effective uncertainty quantification for systems governed by hyperbolic PDEs.  Specifically, we provide a predictor-corrector structure for guiding the allocation of degrees of freedom as the stochastic hyperbolic PDE evolves that automates the solution procedure when associated with a desired local, equilibrating error indicator tolerance. Moreover, by leveraging a flux 1-irregularity property, our method supports anisotropic refinements, which tailor to the variation and modeling difficulty in the physical \textit{and} stochastic spaces separately, under a streamlined computational framework. Rather than constrained to inserting DoFs related exponentially to the dimension of the computational space, the support for anisotropic refinements enables more targeted application of computational resources.
\bibliographystyle{siamplain}
\bibliography{references}

\end{document}

%% file: main_includes.tex
\usepackage{lipsum}
\usepackage{amsfonts}
\usepackage{graphicx}
\usepackage{epstopdf}
\usepackage{breqn}
\usepackage{pgfplots}
\pgfplotsset{grid style = {dotted}, label style={font=\LARGE}}
\usepackage{animate}
\usepgfplotslibrary{fillbetween}

\usepackage{graphicx}
\usepackage{algorithmic}
\ifpdf
  \DeclareGraphicsExtensions{.eps,.pdf,.png,.jpg}
\else
  \DeclareGraphicsExtensions{.eps}
\fi

\usepackage{enumitem}
\setlist[enumerate]{leftmargin=.5in}
\setlist[itemize]{leftmargin=.5in}

\newsiamremark{remark}{Remark}
\newsiamremark{hypothesis}{Hypothesis}
\crefname{hypothesis}{Hypothesis}{Hypotheses}
\newsiamthm{claim}{Claim}

\headers{Adaptive UQ for Hyperbolic Conservation Laws}{J. J. Harmon, S. Tokareva, A. Zlotnik, P. J. Swart}

\title{Adaptive Uncertainty Quantification for Stochastic Hyperbolic Conservation Laws\thanks{Submitted to the editors December 15, 2023.
\funding{This work was supported by the Laboratory Directed Research and Development (LDRD) program at Los Alamos National Laboratory. LA-UR 23-32498.}}}

\author{Jake J. Harmon\thanks{Applied Mathematics and Plasma Physics Group, Theoretical Division, Los Alamos National Laboratory, PO Box 1663, Los Alamos, NM 87545, USA
  (\email{harmon@lanl.gov}, \email{tokareva@lanl.gov}, \email{azlotnik@lanl.gov}, \email{swart@lanl.gov}).}
\and Svetlana Tokareva\footnotemark[2]
\and Anatoly Zlotnik\footnotemark[2]
\and Pieter J. Swart\footnotemark[2]
}
\usepackage{amsopn}

\usepackage{color, soul}
\usepackage{mathtools}
\usepackage{bbm}
\usepackage{caption}
\usepackage{subcaption}
\usepackage{float}

\pgfplotsset{plotstyle/.style={
width=0.6\columnwidth,
height=0.4\columnwidth,
scale only axis,
xminorticks=true,
xlabel style={font=\normalsize},
yminorticks=true,
ylabel style={font=\normalsize},
axis background/.style={fill=white},
xmajorgrids,
xminorgrids,
ymajorgrids,
yminorgrids,
legend style={at={(0.5,1.03)}, anchor=south, legend columns=4, legend cell align=left, align=left, draw=white!15!black}
}
}
\pgfplotsset{compat=1.16}